\documentclass[12pt,a4paper]{amsart}
\usepackage{amsmath,amsthm,amssymb,latexsym,a4wide,epic,eepic,bbm,stmaryrd,comment}

\newtheorem{theorem}{Theorem}
\newtheorem{lemma}[theorem]{Lemma}
\newtheorem{corollary}[theorem]{Corollary}
\newtheorem{proposition}[theorem]{Proposition}

\newtheorem{example}[theorem]{Example}
\newtheorem{remark}[theorem]{Remark}
\usepackage[all]{xy}

\usepackage[parfill]{parskip}

\usepackage{enumerate}
\numberwithin{equation}{section}
\newcommand{\tto}{\twoheadrightarrow}

\font\sc=rsfs10
\newcommand{\cC}{\sc\mbox{C}\hspace{1.0pt}}

\begin{document}
\title{On multisemigroups}
\author{Ganna Kudryavtseva and Volodymyr Mazorchuk}
\date{\today}

\begin{abstract}
Motivated by recent appearance of multivalued structures in categorification, tropical geometry and other areas,
we study basic properties of abstract multisemigroups. We give many new and old examples and general
constructions for multisemigroups. Special attention is paid to simple and nilpotent multisemigroups. We also 
show that ``almost all'' randomly chosen multivalued binary operations define multisemigroups.
\end{abstract}
\maketitle

\section{Introduction and motivation}\label{s1}

A {\em multisemigroup} is a ``semigroup in which multiplication is multivalued''. More precisely, a multisemigroup is
a pair $(S,*)$, where $S$ is a non-empty set and $*$ is a map (a so-called
{\em multivalued operation} or {\em multioperation}) from $S\times S$ to $2^{S}$, the power set of
$S$, such that the following associativity axiom is satisfied: for every $a,b,c\in S$ we have
\begin{equation}\label{eq1}
\bigcup_{t\in b*c}a*t=\bigcup_{s\in a*b}s*c.  
\end{equation}
Every semigroup is a multisemigroup in an obvious way, however, there are many natural examples of multisemigroups 
which are not semigroups. 

Our motivation for the present paper comes from the fact observed in \cite{MM2} that 
multisemigroups appear naturally in higher representation theory and categorification, see Subsection \ref{s3.8} of this paper and also \cite{MM2} for more details. 

Definition of multistructures in general goes back at least to the 1934 paper \cite{Ma} by Marty. Various aspects
and properties of multistructures, in particular, multisemigroups (usually under different names which will be
reviewed in the next section), were studied by several authors, see for example \cite{Ca,DO,Ea,Ha,Ko,Wa}. We also 
refer the reader to the recent survey paper \cite{Vi} by Viro, which mostly deals with multigroups and multirings, 
for more historical information. The paper \cite{Vi} shows that some multistructures (mainly multifields) are 
relevant to the study of tropical geometry. 

In the present paper we make an attempt to
establish basic properties of abstract multisemigroups in analogy with those of semigroups. We give 
many examples and counterexamples illustrating the scopes of the theory we are developing. We will see that in many respects semigroups and 
multisemigroups are similar but in many other respects they are very different. Section~\ref{s2} fixes notation
and vocabulary for multivalued analogues of binary operations and corresponding structures, including multisemigroups, hypergroups 
and multigroups. In Section~\ref{s3} we propose many different examples and constructions of multisemigroups. These
include both concrete examples of multisemigroups of small cardinality (Subsections~\ref{s3.1}, \ref{s3.2} 
and \ref{s3.15}), the multisemigroup of cosets of a group modulo a  subgroup 
(Subsection~\ref{s3.3}) and various constructions associated to different kinds of ideals, variants, subwords 
etc. (Subsections~\ref{s3.4}, \ref{s3.5}, \ref{s3.6}, \ref{s3.11}, \ref{s3.12}  and \ref{s3.14}).
We also mention several multisemigroups related to our motivation: the multisemigroup associated to a fully
additive bicategory (Subsection~\ref{s3.8}) and its disguise via a positive basis in an associative algebra
(Subsection~\ref{s3.7}). We also show that the underlying set of every Weyl group can be equipped with two very 
different multisemigroups structures: the first one coming from the Kazhdan-Lusztig combinatorics 
(Subsection~\ref{s3.9}) and the other one coming from the standard basis of the corresponding Hecke algebra for the 
values of the parameter which guarantee that this standard basis has necessary positivity 
properties (Subsection~\ref{s3.10}). 

In Section~\ref{s4} we collect basic notions and properties: Green's relations, various types of elements, ideals,
homomorphisms, congruences, quotients, representations by binary relations and connections to dioids, quantales
and Boolean algebras.
In Section~\ref{s5} we study simple multisemigroups (that is multisemigroups with a unique two-sided ideal) and
a special class of simple multisemigroups which we call {\em strongly} simple. A simple multisemigroup is
strongly simple provided that it is a union of its minimal left ideals and it is a union of its minimal right ideals.
Strongly simple multisemigroups can be viewed as analogues of completely $0$-simple or $0$-bisimple semigroups. For strongly simple multisemigroups we establish a structure theory similar to the classical structure 
theory of $0$-bisimple semigroups. However, there are significant differences: The role of idempotent
$\mathcal{H}$-classes is now played by hypergroups (and the latter do not need to have any idempotents or 
quasi-idempotents, see an example in Subsection~\ref{s3.15}). There is only a partial analogue of Green's lemma,
namely minimal right (left) ideals are no longer connected by bijective translations but by surjective
multivalued translations. We give explicit examples showing that the general case behaves much more complicated
than the case of classical semigroups (see Subsections~\ref{s5.205}, \ref{s5.207}, \ref{s5.31}, \ref{s5.209}  and \ref{s5.3}). 

In Section~\ref{s6} we establish another essential difference between semigroups and multisemigroups. 
Given a set with $n$ elements, one can randomly choose a binary operation on it and ask what is the
probability that it is associative (hence defining a semigroup). The answer is that this probability is
``small'' in the sense that it tends to $0$ when $n$ tends to infinity. For multisemigroups the picture turns
out to be exactly the opposite: when $n$ tends to infinity, then the probability that a randomly chosen
multivalued operation on an $n$-element set defines a multisemigroups approaches $1$.

Finally, in Section~\ref{s7} we look closer at nilpotent multisemigroups, characterize them in terms of the action 
graphs, extend to finite multisemigroups two classical characterizations of finite nilpotent semigroups, define
and characterize the radical of a multisemigroup and study maximal (with respect to inclusions) nilpotent 
submultisemigroups of strongly simple multisemigroups.
\vspace{0.5cm}

\noindent
{\bf Acknowledgment.} This work was done during the visit of the first author to Uppsala University, which was 
supported by the Department of Mathematics, Uppsala University and the Swedish Research Council. Hospitality and
support of Uppsala University is gratefully acknowledged. The research of the first author was also partially 
supported by the ARRS grant P1-0288. For the second author the research was partially supported by the 
Royal Swedish Academy of Sciences and the Swedish Research Council. We thank Denis Saveliev for useful
comments and discussions. We are really thankful to the referees for very careful reading of the paper, many
very helpful suggestions and for pointing out several subtle inaccuracies in the original version of the paper.

\section{Multistructures and their aliases}\label{s2}

\subsection{Multisemigroups}\label{s2.1}

Multisemigroups, as defined above, have appeared in the literature under many different names. In particular, we
have seen the following names: hypersemigroups, polysemigroups, semihypergroups and  associative multiplicative systems 
(the list disregards different hyphenations of the names). Following \cite{MM2}, we use the name ``multisemigroups'' 
as we think that, compared to all other aliases, it describes the essence of the structure best. Note that the 
operation $*$ of a multisemigroup $(S,*)$ can also be understood as a binary relation from $S\times S$ to $S$ 
(satisfying the usual associativity axiom).

An element $1$ of a multisemigroup $(S,*)$ is called an {\em identity} or {\em unity} 
element provided that $1*a=a*1=a$ for all 
$a\in S$. A standard argument shows that the identity element is unique, if exists. For any multisemigroup $(S,*)$
and any element $e\not\in S$, the set $S\cup\{e\}$ has the natural structure 
$(S\cup\{e\},\circ)$ of a multisemigroup defined, for $a,b\in S\cup\{e\}$, as follows:
\begin{displaymath}
a\circ b:=\begin{cases}a*b,&a,b\in S;\\a,& b=e;\\b,&a=e.\end{cases} 
\end{displaymath}
The element $e$ is the identity of $(S\cup\{e\},\circ)$. As usual, we denote by $S^1$ the multisemigroup 
$S$ if the latter has an identity and $S\cup\{e\}$ defined as above if $S$ does not have any identity
(and we denote this identity by $1$ and  the operation on $S^1$ by $*$).

An element $e$ of a multisemigroup $(S,*)$ is called an {\em idempotent} provided that $e*e=e$ and
a {\em quasi-idempotent} provided that $e\in e*e$.

Multisemigroups are closely connected to some other {\em different} algebraic structures. Here we describe two of 
such connections. A map $*:S\times S\to 2^S$, $(a,b)\mapsto a*b$, can be extended to a binary operation
$*:2^S\times 2^S\to 2^S$ by setting, for $A,B\in 2^S$,
\begin{equation}\label{eq2}
A*B:=\bigcup_{a\in A,\, b\in B} a*b.
\end{equation}
The equality \eqref{eq1} can now be written as $a*(b*c)=(a*b)*c$. In this way, for a multisemigroup 
$(S,*)$, the power set $2^S$ inherits a natural structure of a semigroup, where the associativity follows
directly from \eqref{eq1}. Thus, a multisemigroup can be seen as a non-empty set $S$ along with an 
associative binary operation $*$ on $2^S$ satisfying \eqref{eq2} for any $A,B\subset S$. Notice that to show that
such a binary operation $*$ on $2^S$ defines a multisemigroup on $S$ it is enough to show that 
$a*(b*c)=(a*b)*c$ for all $a,b,c\in S$. From \eqref{eq2} it follows that 
for any $A,B_i\in 2^S$, $i\in I$, we have the following property: 
\begin{equation}\label{eq3}
A*(\cup_i B_i)=\cup_i (A*B_i)\quad\text{ and }\quad (\cup_i B_i)*A=\cup_i (B_i*A). 
\end{equation}

Recall, see e.~g. \cite[2.2]{Gu}, that a {\em semiring} is a tuple $(R,+,*,0,1)$ where $R$ is a non-empty set, $+$ and 
$*$ are binary operations on $R$ and $0$ and $1$ are two distinguished elements of $R$ such that the following 
conditions are satisfied:
\begin{itemize}
\item $(R,+,0)$ is a commutative monoid with identity $0$;
\item $(R,*,1)$ is a monoid with identity $1$;
\item $r*0=0*r=0$ for all $r\in R$ (i.e. $0$ is {\em absorbing} with respect to $*$);
\item $r*(s+t)=r*s+r*t$ and $(s+t)*r=s*r+t*r$ for all $r,s,t\in R$.
\end{itemize}
A semiring $(R,+,*,0,1)$ for which $r+r=r$ for all $r\in R$ is called an {\em idempotent semiring} 
or {\em dioid}. The previous paragraph implies that any multisemigroup $(S,*)$ induces a natural dioid
structure $(2^{S^1},\cup,*,\varnothing,\{1\})$ on $2^{S^1}$. If $Q$ is finite, then every dioid
structure $(2^{Q},\cup,*,\varnothing,1)$, where $1$ is a singleton, gives, by restriction to elements of $Q$, a multisemigroup $(Q,*)$ possessing a unit element.

A notable difference between multisemigroups  and dioids is revealed comparing the arbitrary distributivity
property for multisemigroup given by \eqref{eq3} with the finite distributivity property for dioids. This 
discrepancy motivates connection of multisemigroups with quantales. 

Recall, see \cite{Re}, that a {\em sup-lattice} is a partially ordered set in which every subsets has a {\em supremum}, 
or a {\em join}. A {\em quantale} is a sup-lattice equipped with an associative product $(a,b)\mapsto a*b$ such that 
the multiplication distributes over arbitrary suprema, that is
for any $X\subset Q$ and $a\in Q$ we have 
\begin{displaymath}
a*(\mathrm{sup}_{x\in X}x)=\mathrm{sup}_{x\in X}(a*x)  \text{ and }
(\mathrm{sup}_{x\in X}x)*a=\mathrm{sup}_{x\in X}(x*a). 
\end{displaymath}
A {\em unital quantale} is a quantale possessing a unit element with respect to the product operation. Unital 
quantales form a special class of dioids, the so-called {\em complete dioids}. 

The discussion above shows that any multisemigroup $(S,*)$ induces a natural quantale structure 
on $2^{S}$. Being defined on a power set, this quantale is a {\em complete atomic Boolean algebra}
where atomic means that every element lies above an atom, a minimal nonzero element. Conversely, 
every quantale structure on a complete atomic Boolean algebra $Q$ induces a natural structure of 
a multisemigroup on the set $S=S(Q)$ of atoms of $Q$. This can be used to identify multisemigroups and
quantale structures on complete atomic Boolean algebras.

\subsection{Multigroups and hypergroups}\label{s2.2}

Unlike multisemigroups, whose definition is more or less uncontroversial despite of many different aliases,
there exist many {\em different} multistructures considered in the literature as multivalued analogues of groups. 
The present paper is not about these structures, so to fix terminology for the present paper we will just define 
those of them which we will use later. We refer the reader to \cite{DO,Vi,Wa} and references therein for further
details.

Following \cite{Vo}, a multisemigroup $(S,*)$ will be called a {\em hypergroup} provided that it satisfies
the following {\em reproduction axiom}:  $S*a=a*S=S$ for any $a\in S$.

Following \cite{Vi}, a multisemigroup $(S,*)$ with identity $1$ will be called a {\em multigroup} provided 
that the following two conditions are satisfied:
\begin{itemize}
\item for each $a\in S$ there are unique $b,c\in S$ such that $1\in a*b$ and $1\in c*a$, moreover, $b=c$
(this unique element will be denoted by $a^{-1}$);
\item for any $a,b,c\in S$ we have $c\in a*b$ if and only if $c^{-1}\in b^{-1}*a^{-1}$. 
\end{itemize}
It is clear that $1^{-1}=1$ and $(a^{-1})^{-1}=a$ for any element $a$ of a multigroup $S$
(cf. \cite[3.2]{Vi}).

\subsection{Involution}\label{s2.3}

If $(S,*)$ is a multisemigroup, then an {\em involution} on $S$ is a transformation $\star:S\to S$,
written $s\mapsto s^{\star}$, which is involutive, that is $(s^\star)^\star=s$ for all $s\in S$, and
satisfies, for all $a,b\in S$, the following equality:
\begin{displaymath}
b^{\star}*a^{\star}=\{s^{\star}\,\vert\, s\in a*b\}.
\end{displaymath}
For example, if $(S,*)$ is a multigroup, then $a\mapsto a^{-1}$ is an involution.

\section{Some examples of multisemigroups}\label{s3}

In this section we collect many examples of multisemigroups, for further examples of multigroups we refer
the reader to \cite{Vi} and references therein.

\subsection{The trivial multisemigroups}\label{s3.1}

For any non-empty set $S$ and a subset $X\subset S$ setting $s*_Xt:=X$ for all $s,t\in S$ defines on $S$
the structure of a multisemigroup. In particular, we have two {\em trivial} multisemigroup structures on $S$,
namely, $\diamond:=*_\varnothing$ and $\bullet:=*_S$.
We note that $(S,\bullet)$ is always a hypergroup, but it is not a multigroup if $|S|>1$.

\subsection{A two-element multisemigroup}\label{s3.2}

Define the operation $*$ on $S:=\{a,b\}$ using the following Cayley table:
\begin{displaymath}
\begin{array}{c||c|c}
*&a&b\\
\hline\hline
a&a&\{a,b\}\\
\hline
b&a&b
\end{array} 
\end{displaymath}
It is straightforward to verify that this is a multisemigroup. We note that $x*S=S$ for any $x\in S$ while
$S*a=a\neq S$ (in particular, this is a {\em right} hypergroup but not a hypergroup).

\subsection{The coset multisemigroup}\label{s3.3}

Let $(G,\cdot)$ be a group and $H$ a (not necessarily normal) subgroup of $G$. Define a multivalued operation $*$ on $G$ as 
follows: for every $a,b\in G$ we set $a*b:=HaHb$. It is straightforward to verify that $(a*b)*c=a*(b*c)=HaHbHc$ for all
$a,b,c\in G$ and hence $(G,*)$ is a multisemigroup.

We can also consider the set $H\backslash G$ of left $H$-cosets in $G$. Then for any $a,b\in G$ the set 
$HaHb$ is a union of cosets and hence we may define 
\begin{displaymath}
Ha*Hb:=\{Hc\,\vert\, c\in G\text{ and }Hc\subset HaHb\}.
\end{displaymath}
This turns $(H\backslash G,*)$ into a multisemigroup and even a hypergroup (but not a multigroup if $H$ is not normal
since an identity would be necessarily $H$, and $H$ is normal if and only if, for all $a,b\in G$,
$H\subset HaHb$ implies $Hb=Ha^{-1}$). 
If $H$ is a normal subgroup of $G$, then the operation $*$ on $H\backslash G$ is, in fact, single-valued 
and hence $(H\backslash G,*)$ is a group.

\subsection{Inflations of multisemigroups}\label{s3.4}

Let $(S,*)$ be a multisemigroup, $X$ an arbitrary set, and $f:X\to S$ a surjective map. For $x,y\in X$ define
\begin{displaymath}
x*_f y:=\{z\in X\,\vert\, f(z)\in f(x)*f(y) \}. 
\end{displaymath}
Then it is straightforward to verify that $(X,*_f)$ is a multisemigroup called the {\em inflation} of $S$ with respect 
to $f$. Note that the trivial multisemigroup  $(S,\bullet)$ defined in Subsection \ref{s3.1} can be viewed as an inflation of a singleton group.

\subsection{Multisemigroups of ideals}\label{s3.5}

Let $(S,\cdot)$ be a semigroup. Define multioperations $*_L$ and $\hat{*}_L$ on $S$ as follows: for 
$a,b\in S$ set
\begin{displaymath}
a *_L b:=S^1aS^1b\quad\text{ and }\quad a \hat{*}_L b:=S^1a\cap S^1b.
\end{displaymath}
Then it is straightforward to verify that both $(S,*_L)$ and $(S,\hat{*}_L)$ are multisemigroups.
Similarly one defines multisemigroups $(S,*_R)$ and $(S,\hat{*}_R)$ using right ideals and
multisemigroups $(S,*_{J})$ and $(S,\hat{*}_J)$ using two-sided ideals. The multisemigroups
$(S,\hat{*}_L)$, $(S,\hat{*}_R)$ and $(S,\hat{*}_J)$ are commutative.

\subsection{Monogenic associated multisemigroups}\label{s3.6}

Let $(S,\cdot)$ be a semigroup. For $a\in S$ let $\langle a\rangle$ denote the subsemigroup of $S$ consisting of
all elements of the form $a^i$, $i>0$ (the so-called ``monogenic subsemigroup'' generated by $a$).
Define the multioperation $*$ on $S$ as follows: for  $a,b\in S$ set
\begin{displaymath}
a * b:=\langle a\rangle\cap \langle b\rangle.
\end{displaymath}
Then it is straightforward to verify that $(S,*)$ is a commutative multisemigroup.

\subsection{Multisemigroups of positive bases in associative algebras}\label{s3.7}

Let $A$ be an associative algebra over some subring $\Bbbk$ of real numbers. Assume that  $A$ has a basis 
$\mathbf{a}:=\{a_i\,\vert\, i\in S\}$ with non-negative structure constants, that is
\begin{displaymath}
a_ia_j=\sum_{k\in S}c_{i,j}^k a_k \quad\text{ and }\quad c_{i,j}^k\geq 0\quad\text{ for all }\quad i,j,k\in S. 
\end{displaymath}
Define the multioperation $*$ on $S$ as follows: for  $i,j\in S$ set
\begin{displaymath}
i * j:=\{k\,\vert\, c_{i,j}^k>0\}.
\end{displaymath}
Then the associativity of $A$ implies that $(S,*)$ is a multisemigroup.

A similar construction works if instead of a subring of real numbers one considers, for example, the Boolean
algebra $\mathbb{B}:=\{0,1\}$ (with respect to the usual meet and join operations).

\subsection{Multisemigroups of fully additive bicategories}\label{s3.8}

This example is taken from \cite[Subsection~3.3]{MM2}. Let $\cC$ be a small additive bicategory with skeletally 
small, fully additive and Krull-Schmidt categories 
of morphisms. Let $S[\cC]$ be the set of isomorphism classes of indecomposable $1$-morphisms in $\cC$. 
For an indecomposable $1$-morphism $\mathrm{F}$ we denote by $[\mathrm{F}]$ its class in $S[\cC]$. 
For $[\mathrm{F}],[\mathrm{G}]\in S[\cC]$ set
\begin{displaymath}
[\mathrm{F}]*[\mathrm{G}]:=\{[\mathrm{H}]\in S[\cC]\,:\, \mathrm{H}\text{ is isomorphic to a direct summand of }
\mathrm{F}\circ\mathrm{G}\}.
\end{displaymath}
Then the associativity axiom for $\cC$ implies that $(S[\cC],*)$ is a multisemigroup. Via decategorification
(i.e. taking the split Grothendieck group of $\cC$) this example can be considered as a special case of 
the previous example.

As a more concrete example of this construction, consider $\cC$ to be a bicategory with one object $\bullet$
and such that $\cC(\bullet,\bullet)$ is the category of all finite dimensional representations of a semi-simple 
complex finite dimensional Lie algebra $\mathfrak{g}$, with horizontal composition given by the usual tensor product 
of $\mathfrak{g}$-modules. It is easy to see that in this case the obtained multisemigroup is, in fact, a multigroup. 
In the case of the algebra $\mathfrak{sl}_2$, isomorphism classes of simple finite dimensional modules are in a 
natural bijection with the set  $\mathbb{N}_0$ of non-negative integers (this bijection is given by taking 
the highest weight of a module, see \cite[Theorem~1.22]{Maz2}). From the classical Clebsch-Gordan rule 
(see e.g. \cite[Theorem~1.39]{Maz2}) it follows that for $k,l\in\mathbb{N}_0$ the multisemigroup 
operation is given by the following:
\begin{displaymath}
k*l=\{m\,:\, |k-l|\leq m\leq k+l, m\equiv k+l\text{ mod }2\}.
\end{displaymath}

\begin{remark}\label{remnn1}
{\rm 
The concrete example above can be generalized. For $\mathfrak{g}$ as above  consider the BGG category 
$\mathcal{O}$ and let $\overline{\mathcal{O}}$ be the tensor category which $\mathcal{O}$ generates 
(see e.g. \cite{Kaa}). The category $\overline{\mathcal{O}}$ is no longer semi-simple but all objects
in $\overline{\mathcal{O}}$ have well-defined composition multiplicities (see e.g. \cite{Kaa}).
There is a natural  multisemigroup structure on the set of isomorphism classes of simple objects in
$\overline{\mathcal{O}}$ defined as follows: For two simple objects $L$ and $L'$ define $[L]*[L']$ to 
be the set of isomorphism classes of all simple subquotients which appear in the tensor product of $L$ and $L'$. 
}
\end{remark}

\subsection{The Kazhdan-Lusztig multisemigroup of a Weyl group}\label{s3.9}

Let $\Delta$ be a finite root system and $W$ the corresponding Weyl group. Let $\{\underline{H}_w\,:\,w\in W\}$
be the Kazhdan-Lusztig basis of $\mathbb{Z}[W]$, see \cite{KL} or \cite[Section~7]{Maz}. By \cite{KL}, this basis 
has positive structure constants and hence the construction of Subsection~\ref{s3.7} gives a multisemigroup
structure $(W,*)$. Remark that this multisemigroup can be also obtained by the construction of Subsection~\ref{s3.8}
considering the bicategory of Soergel bimodules acting on the principal block of category $\mathcal{O}$
for the Lie algebra $\mathfrak{g}$ associated with $\Delta$, see \cite[Section~8]{Maz} and \cite{MM1,MM2} for details.
Mapping $w\mapsto w^{-1}$ defines an involution on this multisemigroup (it corresponds to ``taking the adjoint functor''
in the categorical picture of Soergel bimodules).

To give an explicit example, let $\Delta$ be of type $A_2$, that is $W\simeq S_3=\{e,s,t,st,ts,sts\}$, where
$s^2=t^2=e$ and $sts=tst$. Then we have
\begin{gather*}
\underline{H}_e=e,\,\,\, \underline{H}_s=e+s,\,\,\, \underline{H}_t=e+t,\,\,\,\underline{H}_{st}=e+s+t+st,\\ 
\underline{H}_{ts}=e+s+t+ts,\,\, \,  \underline{H}_{sts}=e+s+t+st+ts+sts;
\end{gather*}
and one easily obtains the following Cayley table for $(S_3,*)$:
\begin{displaymath}
\begin{array}{c||c|c|c|c|c|c}
*&e&s&t&st&ts&sts\\
\hline\hline 
e&e&s&t&st&ts&sts\\
\hline
s&s&s&st&st&\{sts,s\}&sts\\
\hline
t&t&ts&t&\{sts,t\}&ts&sts\\
\hline
st&st&\{sts,s\}&st&\{sts,st\}&\{sts,s\}&sts\\
\hline
ts&ts&ts&\{sts,t\}&\{sts,t\}&\{sts,ts\}&sts\\
\hline
sts&sts&sts&sts&sts&sts&sts\\
\end{array}
\end{displaymath}

\subsection{The Boolean Hecke multigroup of a Weyl group}\label{s3.10}

This example is taken from \cite{Tr} where it is given in different terms. Let $\Delta$
be a finite root system and $W$ the corresponding Weyl group. Choose some basis $\pi$ in $\Delta$ and let
$S\subset W$ be the corresponding system of simple reflections. Then $(W,S)$ is a Coxeter system. Let
$\mathfrak{l}:W\to \mathbb{N}_0$ be the corresponding length function. Fix some 
$q\in \mathbb{R}$ such that $q>1$ and let $\mathbb{H}_q$ be the corresponding Hecke algebra (over $\mathbb{R}$), 
that is the associative algebra with generators $H_s$, $s\in S$, satisfying the braid relations for $(W,S)$ together
with the relations 
\begin{equation}\label{eq4}
H_s^2=(q-1)H_s+qH_e,\quad s\in S. 
\end{equation}
Note that under our choice of $q$ the latter relation has positive coefficients.

For each $w\in W$ fix some reduced expression $w=s_1s_2\cdots s_k$ and set $H_w:=H_{s_1}H_{s_2}\cdots H_{s_k}$.
Since the $H_s$'s satisfy braid relations, the element $H_w$ does not depend on the choice of a reduced expression for 
$w$.  Then $\{H_w\,\vert\,w\in W\}$ is the {\em standard} basis of $\mathbb{H}_q$. 

\begin{lemma}\label{lem1}
All structure constants for the standard basis are non-negative. 
\end{lemma}

\begin{proof}
As $H_w:=H_{s_1}H_{s_2}\cdots H_{s_k}$, it is enough to show that for any $s\in S$ and $x\in W$ the element
$H_sH_x$ is a linear combination of basis elements with non-negative coefficients. If $\mathfrak{l}(sx)>
\mathfrak{l}(x)$, then we have $H_sH_x=H_{sx}$. In the other case we have $x=sy$ for some $y\in W$ such that
$\mathfrak{l}(y)<\mathfrak{l}(x)$. Then, using \eqref{eq4}, we have
\begin{displaymath}
H_sH_x= H_sH_sH_y=\big((q-1)H_s+qH_e\big)H_y=(q-1)H_x+ q H_y
\end{displaymath}
and the claim follows.
\end{proof}

From Lemma~\ref{lem1} it follows that the construction of Subsection~\ref{s3.7} gives a multisemigroup
structure $(W,*)$. This multisemigroup is called the {\em Boolean Hecke hypermonoid} in \cite{Tr}.
If $\Delta$ is of type $A_2$, then for $W\simeq S_3=\{e,s,t,st,ts,sts\}$ we get the following  Cayley table:
\begin{displaymath}
\begin{array}{c||c|c|c|c|c|c}
*&e&s&t&st&ts&sts\\
\hline\hline 
e&e&s&t&st&ts&sts\\
\hline
s&s&\{e,s\}&st&\{t,st\}&sts&\{ts,sts\}\\
\hline
t&t&ts&\{e,t\}&sts&\{s,ts\}&\{st,sts\}\\
\hline
st&st&sts&\{s,st\}&\{ts,sts\}&\{e,s,sts\}&\{t,st,ts,sts\}\\
\hline
ts&ts&\{t,ts\}&sts&\{e,t,sts\}&\{st,sts\}&\{s,st,ts,sts\}\\
\hline
sts&sts&\{st,sts\}&\{ts,sts\}&\{s,st,ts,sts\}&\{t,st,ts,sts\}&\{e,s,t,st,ts,sts\}\\
\end{array}
\end{displaymath}

\begin{proposition}\label{prop2}
The multisemigroup  $(W,*)$ is, in fact, a multigroup. 
\end{proposition}

\begin{proof}
We have to check that both additional conditions from Subsection~\ref{s2.2} are satisfied.
Let $x\in W$. First we show, by induction on $\mathfrak{l}(x)$, that $e\in x*x^{-1}$. This is clear if $x=e$, so to 
prove the induction step assume that the claim is true for some $x$ and that $s\in S$ is such that $\mathfrak{l}(sx)>
\mathfrak{l}(x)$. Then, using associativity and definitions, we have $(sx)*(sx)^{-1}=s*x*x^{-1}*s$. The latter
set contains $s*e*s=\{e,s\}$ as $e\in x*x^{-1}$ by the inductive assumption.

Now let $x,y\in W$ be such that $x\neq y^{-1}$. Consider the case $\mathfrak{l}(x)\leq \mathfrak{l}(y)$, in particular,
$y\neq e$ (the other case is dealt with by similar arguments). Let us prove, by induction on $\mathfrak{l}(x)$, 
the following three claims: 
\begin{enumerate}[(i)]
\item\label{claim1} any $w\in x*y$ satisfies $\mathfrak{l}(w)\geq \mathfrak{l}(y)-\mathfrak{l}(x)$;
\item\label{claim2} $x*y$ contains some $w$ satisfying $\mathfrak{l}(w)=\mathfrak{l}(y)-\mathfrak{l}(x)$
if and only if $y=x^{-1}w$;
\item\label{claim3} $e\not\in x*y$.
\end{enumerate}
Note that the last claim obviously follows from the first two.
The basis $x=e$ of the induction is obvious. To prove the induction step let $s\in S$ be a simple reflection such that
$\mathfrak{l}(x)<\mathfrak{l}(sx)\leq \mathfrak{l}(y)$. Then we have $sx=s*x$ and hence $(sx)*y=s*(x*y)$ by 
associativity. By the inductive assumption, any $w\in x*y$ satisfies 
$\mathfrak{l}(w)\geq \mathfrak{l}(y)-\mathfrak{l}(x)$. We have $s*w=sw$ if $\mathfrak{l}(sw)>\mathfrak{l}(w)$
and $s*w=\{w,sw\}$ otherwise. In the second case $\mathfrak{l}(sw)\geq \mathfrak{l}(y)-\mathfrak{l}(x)-1=
\mathfrak{l}(y)-\mathfrak{l}(sx)$ which proves claim \eqref{claim1}. This argument also shows that if
$k$ is maximal such that $x*y$ contains some $w$ of length $\mathfrak{l}(y)-k$, then every element in
$s*(x*y)$ has length at least $\mathfrak{l}(y)-k-1$. Therefore, for $s*(x*y)$ to contain an element of length
$\mathfrak{l}(y)-\mathfrak{l}(sx)$, the set $x*y$ must contain an element of length 
$\mathfrak{l}(y)-\mathfrak{l}(x)$. By \eqref{claim2} of the inductive assumption, the only element of $x*y$ with 
this property is $xy$ and this is possible if and only if  $\mathfrak{l}(xy)=\mathfrak{l}(y)-\mathfrak{l}(x)$. 
In the latter case $s*xy$ contains an element of length $\mathfrak{l}(y)-\mathfrak{l}(x)-1$ if and only
if $\mathfrak{l}(sxy)=\mathfrak{l}(y)-\mathfrak{l}(x)-1$ and this is the case if and only if 
$\mathfrak{l}(sxy)=\mathfrak{l}(y)-\mathfrak{l}(sx)$ which implies claim \eqref{claim2}.
This proves the first condition from Subsection~\ref{s2.2}.

Mapping $H_s\mapsto H_{s}$ extends to an anti-involution on $\mathbb{H}_q$ (as the ideal generated by the defining
relations is invariant under this map). This anti-involution maps $H_w$ to $H_{w^{-1}}$ for any $w\in W$ and hence
mapping $w\mapsto w^{-1}$ is an involution on $(W,*)$. This proves the second condition from Subsection~\ref{s2.2}.
\end{proof}

\subsection{Double variants of multisemigroups}\label{s3.11}

Let $(S,\cdot)$ and $(S,\circ)$ be two multisemigroups with the same underlying set $S$.  Assume further that for any $a,b,c\in S$ we have the following equalities in $2^S$:
\begin{equation}\label{eq5}
(a\cdot b)\circ c= a\cdot (b\circ c)\quad\text{ and }\quad
(a\circ b)\cdot c= a\circ (b\cdot c). 
\end{equation}
For $a,b\in S$ set $a*b:=(a\cdot b)\cup(a\circ b)$.

\begin{proposition}\label{prop3}
$(S,*)$ is a multisemigroup.  
\end{proposition}

\begin{proof}
For $a,b,c\in S$ we have the following:
\begin{displaymath}
\begin{array}{rcl}
(a*b)*c&=&\big((a\cdot b)\cup(a\circ b)\big)*c\\
&=&\big((a\cdot b)*c\big)\cup\big((a\circ b)*c\big)\\
&=&\big((a\cdot b)\cdot c\big)\cup\big((a\cdot b)\circ c\big)\cup\big((a\circ b)\cdot c\big)\cup
\big((a\circ b)\circ c\big)\\
&\overset{\eqref{eq5}}{=}&\big(a\cdot (b\cdot c)\big)\cup\big(a\cdot (b\circ c)\big)\cup
\big(a\circ (b\cdot c)\big)\cup \big(a\circ (b\circ c)\big)\\
&=&\big(a\cdot (b\cdot c)\big)\cup\big(a\circ (b\cdot c)\big)\cup
\big(a\cdot (b\circ c)\big)\cup \big(a\circ (b\circ c)\big)\\
&=&\big(a*(b\cdot c)\big)\cup\big(a*(b\circ c)\big)\\
&=&a*\big((b\cdot c)\cup(b\circ c)\big)\\
&=& a*(b*c).
\end{array}
\end{displaymath}
The claim follows.
\end{proof}

A typical situation in which this construction applies is the following multisemigroup version of the {\em variant}
(or {\em sandwich}) construction for semigroups: Let $(S,\bowtie)$ be a multisemigroup and $X,Y\subset S$. Then for $a,b\in S$ set
\begin{displaymath}
a\cdot b:=a\bowtie X\bowtie b,\quad \text{ and }\quad a\circ b:=a\bowtie Y\bowtie b.
\end{displaymath}
Then both $(S,\cdot)$ and $(S,\circ)$ are {\em variant} multisemigroups of $(S,\bowtie)$, moreover, condition
\eqref{eq5} is obviously satisfied. Thus $(S,*)$ is a new multisemigroup which is natural to call a
{\em double variant} of $(S,\bowtie)$. It is easy to see that
$a*b= a \bowtie (X\cup Y) \bowtie b$. In the case when $(S,\bowtie)$ is a semigroup and $\vert X\vert =\vert Y\vert =1$ we have that $(S,*)$ is a multisemigroup such that $1\leq \vert a*b\vert\leq 2$ for all $a,b\in S$.

\subsection{Multisemigroup of subwords}\label{s3.12}

Let $A$ be an alphabet and $A^*$ the monoid of all finite words over $A$. For $u,v\in A^*$ define
$u\circledast v$ to be the set of all scattered (that is, not necessarily connected) subwords of $uv$. It is straightforward 
to verify that for any $u,v,w\in A^*$ both $(u\circledast v)\circledast w$ and $u\circledast (v\circledast w)$
coincide with the set of all scattered subwords of $uvw$ and hence $(A^*,\circledast)$ is a multisemigroup.

\subsection{Disconnected unions of multisemigroups}\label{s3.13}

Let $(S,\cdot)$ and $(T,\bullet)$ be multisemigroups and assume that $S\cap T=\varnothing$. Define a multivalued
operation $*$ on $S\cup T$ as follows: for $a,b\in S\cup T$ set
\begin{displaymath}
a* b:=\begin{cases}a\cdot b,&a,b\in S;\\a\bullet b,&a,b\in T;\\\varnothing,&\text{otherwise}.\end{cases}
\end{displaymath}
It is straightforward to verify that this turns $(S\cup T,*)$ into a multisemigroup, which we call the
{\em disconnected union} of $S$ and $T$. 
                                         
\subsection{Reproductive construction}\label{s3.14}

The following general approach to construction of multisemigroups is inspired by \cite{Vo}. Let $(S,\cdot)$ 
be a semigroup and $f:S\to 2^S$ a map. For $A\subset S$ set $f(A):=\cup_{a\in A}f(a)$.
For $a,b\in S$ define $a*b:=f(a)f(b)$. 

\begin{lemma}\label{lem11}
Assume that for any $a,b\in S$ we have $f(f(a)f(b))=f(a)f(b)$. Then 
$(S,*)$ is a multisemigroup. 
\end{lemma}

The condition $f(f(a)f(b))=f(a)f(b)$ resembles the {\em reproductive condition} in \cite{Vo}.

\begin{proof}
Using our assumption, for $a,b,c\in S$ we have:
\begin{displaymath}
(a*b)*c=\bigcup_{s\in f(a)f(b)}\big(f(s)f(c)\big)=\left(\bigcup_{s\in f(a)f(b)}f(s)\right)f(c)=
f\big(f(a)f(b)\big)f(c)=f(a)f(b)f(c).
\end{displaymath}
Similarly one checks that $a*(b*c)=f(a)f(b)f(c)$.
\end{proof}

Some of the previous examples can be obtained using the reproductive construction. For instance,
the example in Subsection~\ref{s3.12} is obtained if we define $f$ to be the map which sends a word 
$w$ to the set of all scattered subwords of $w$; the first example in Subsection~\ref{s3.3} is obtained
if we define $f$ to be the map which sends $a$ to $Ha$; and the first example in Subsection~\ref{s3.5} 
is obtained if we define $f$ to be the map which sends $a$ to $S^1a$.

\subsection{A hypergroup without quasi-idempotents}\label{s3.15}

Let $S$ be a set satisfying $|S|\geq 3$. For $a,b\in S$ we define
\begin{displaymath}
a*b:=\begin{cases}S,&a\neq b;\\S\setminus\{a\},&a=b.\end{cases} 
\end{displaymath}
Then $|a*b|> 1$ for any $a,b\in S$, which implies $(a*b)*c=a*(b*c)=S$ for all $a,b,c\in S$. 
Thus $(S,*)$ is a multisemigroup. Obviously, $S$ is a hypergroup and it does not contain any quasi-idempotent.

\subsection{Multisemigroups of ultrafilters}\label{s3.25}

Multisemigroups arise in the recent paper \cite{GGP}. Let $X$ be a set. A {\em filter} on $X$ is a filter 
of the Boolean algebra $2^X$, that is a non-empty collection $F$ of subsets of $X$ such that
\begin{enumerate}
\item $\varnothing\not\in F$.
\item If $A\in F$ and $B\supseteq A$ then $B\in F$.
\item If $A,B\in F$ then $A\cap B\in F$.
\end{enumerate}

An {\em ultrafilter} is a filter which is maximal with respect to inclusion. Denote by $\beta X$ the set of 
ultrafilters on $X$. From Zorn's lemma it follows that any filter is contained in some ultrafilter. 
If some filter is not an ultrafilter, this filter is contained in more than one ultrafilter.

Assume that $M$ is a monoid. 

\begin{lemma} Let $p,q\in \beta M$. Then the set  
$$\{XY: X\in p, Y\in q\}^{\uparrow}=\{A\in 2^M: A\supset XY \text{ for some } X\in p, Y\in q\}$$ 
is a filter that we call the {\em filter of supersets of} $pq$.
\end{lemma}

\begin{proof} Let $X,P\in p$ and $Y,Q\in q$. Then $X\cap P\in p$, $Y\cap Q\in q$ and $(X\cap P)(Y\cap Q)\subset XY\cap PQ$.  It follows that $XY\cap PQ\in \{XY: X\in p, Y\in q\}^{\uparrow}$. The other axioms are immediate to check.
\end{proof}

Following \cite[Example 3.1]{GGP}, define a multivalued operation $\circ$ on $\beta M$ as follows.  
$$
p\circ q= \{f\in \beta M\colon f\supseteq \{XY: X\in p, Y\in q\}^{\uparrow}\}.
$$

\begin{proposition} $(\beta M,\circ)$ is a multisemigroup.
\end{proposition}

\begin{proof}
Let $p,q,r\in\beta M$. Let $S=\{XYZ: X\in p, Y\in Q, Z\in r\}^{\uparrow}$.
It is routine to verify that both $(p\circ q)\circ r$ and $p\circ (q\circ r)$ equal to the set of ultrafilters containing the filter $S$.
\end{proof}

One can similarly verify that, in general, the ternary relation that arises in \cite[Section 3]{GGP}, defines a multivalued multiplication. In particular, \cite[Examples 3.2 and 4.1]{GGP} are examples of multisemigroups.

\section{Elementary properties of multisemigroups}\label{s4}

In this section we provide multisemigroup analogues of some basic notions from semigroup theory as well as record some basic properties of  multisemigroups.

\subsection{Ideals and Green's relations}\label{s4.1}

Let $(S,*)$ be a multisemigroup. A subset $I\subset S$ is called a {\em left ideal} (resp. {\em right ideal}, 
{\em two-sided ideal}) provided that for any $a\in I$ and $s\in S$ we have $s*a\subset I$ (resp. $a*s\subset I$;
$a*s,s*a\subset I$). For example, for every $a\in S$ the set $S^1*a$ is the smallest left ideal containing 
$a$, called the 
{\em principal left ideal} generated by $a$. Similarly one has the {\em principal right ideal} $a*S^1$ and the
{\em principal two-sided ideal} $S^1*a*S^1$. We define the left pre-order $\leq _L$, the right pre-order
$\leq_R$ and the two-sided pre-order $\leq_J$ on $S$ as follows: for $a,b\in S$ set $b\leq_L a$ if and only if
$S^1*b\subset S^1*a$, $b\leq_R a$ if and only if $b*S^1\subset a*S^1$, and $b\leq_J a$ if and only if 
$S^1*b*S^1\subset S^1*a*S^1$. 

Following \cite{Gr}, we define an equivalence relation $\mathcal{L}$ on $S$ as the equivalence relation
induced by $\leq_L$, i.e. $a\,\mathcal{L}\,b$ if and only if $a\leq_L b$ and
$b\leq_L a$. Similarly we define relations $\mathcal{R}$ and $\mathcal{J}$ (see also \cite{Ha,MM2}). We set $\mathcal{H}:=\mathcal{L}\cap\mathcal{R}$ and denote
by $\mathcal{D}$ the minimal equivalence relation containing both $\mathcal{L}$ and $\mathcal{R}$. 
The relations $\mathcal{L}$, $\mathcal{R}$, $\mathcal{J}$, $\mathcal{H}$ and $\mathcal{D}$ are called
{\em Green's relations}. Obviously,
$\mathcal{D}\subset \mathcal{J}$. Note that the equality ${\mathcal L}\circ {\mathcal R}= {\mathcal R}\circ {\mathcal L}$ (where $\circ$ denotes the usual product of binary relations) which holds for any semigroup fails for multisemigroups
in general, see Subsection \ref{s5.31}. For an element $a\in S$ we denote by $\mathcal{L}_a$ the $\mathcal{L}$-class
of $S$ containing $a$. We define $\mathcal{R}_a$, $\mathcal{H}_a$, $\mathcal{J}_a$ and $\mathcal{D}_a$ similarly.

\begin{example}\label{ex15}
{\rm 
Let $(S,\cdot)$ be a semigroup and $(S,*_L)$ be the corresponding multisemigroup of left ideals defined in
Subsection~\ref{s3.5}. For $a\in S$ we have $S^1a=S^1*_L a$ and hence the $\mathcal{L}$-classes in $(S,\cdot)$ and 
in $(S,*_L)$ coincide. Any two-sided ideal of $(S,\cdot)$ is, in particular, a left ideal, which implies that
the $\mathcal{J}$-classes in $(S,\cdot)$ and in $(S,*_L)$ coincide as well. On the other hand, for any
$a\in S$ we have $a*_L S^1=S^1(aS^1)=S^1*_L a*_L S^1$, which implies that the relations $\mathcal{R}$
and $\mathcal{J}$ for $(S,*_L)$ coincide with $\mathcal{J}$ in $(S,\cdot)$. 
Hence we also have $\mathcal{D}=\mathcal{J}$ and $\mathcal{L}=\mathcal{H}$
for the multisemigroup $(S,*_L)$.
}
\end{example}

\subsection{Rees quotients}\label{s4.3}

Let $(S,*)$ be a multisemigroup and $I\subset S$ a two-sided ideal different from $S$ (possibly empty). 
Similarly to 
Subsection~\ref{s4.2}, consider the set $T:=S\setminus I$ and for $a,b\in T$ set $a\bullet b:=(a*b)\setminus I$.
It is straightforward to verify that this turns $(T,\bullet)$ into a  multisemigroup. Making a parallel with 
the classical semigroup theory, we will call the multisemigroup  $(T,\bullet)$ the 
{\em Rees quotient} of $S$ modulo the ideal $I$.

\subsection{Zero elements}\label{s4.2}

Let $(S,*)$ be a multisemigroup. An element $z\in S$ is called a {\em zero} element provided that for every $a\in S$
we have $a*z=z*a=z$. A zero element is necessarily unique, if exists, and therefore it is natural to denote this
unique zero element by $0$. 

Let $(S,*)$ be a multisemigroup with the zero $0$ and suppose that $S\neq \{0\}$. 
Then we claim that for any $a,b\in S$ we have $a*b\neq\varnothing$. Indeed, assume that $a*b=\varnothing$, 
then, on the one hand, $(a*b)*0=\varnothing$, but, on the other hand, $a*(b*0)=a*0=0$, a contradiction. 
Consider the set $T:=S\setminus\{0\}$ and for $a,b\in T$ set $a\bullet b:=(a*b)\setminus\{0\}$. 
It is straightforward to verify that $(T,\bullet)$ is a multisemigroup (see also Subsection~\ref{s4.3}).

Conversely, let $(S,*)$ be a multisemigroup without a zero element. Consider the set $\displaystyle 
S^0:=S\cup\{0\}$, where we assume $0\not\in S$, and for $a,b\in S^0$ define
\begin{displaymath}
a\bullet b:=\begin{cases}(a*b)\cup\{0\};& a,b\in S;\\
\{0\},& \text{otherwise}.\end{cases} 
\end{displaymath}

\begin{lemma}\label{lemzore}
\begin{enumerate}[$($a$)$]
\item\label{lemzore.1} The construct $(S^0,\bullet)$ is a multisemigroup with the zero $0$.
\item\label{lemzore.2} The set $\{0\}$ is an ideal of $(S^0,\bullet)$.
\item\label{lemzore.3} We have $(S^0,\bullet)\setminus\{0\}\cong (S,*)$.
\end{enumerate}
\end{lemma}

\begin{proof}
Let $a,b,c\in S^0$. If one of these elements equals $0$, then both sides
of \eqref{eq1} are equal to $0$. If $a,b,c\in S$, then both sides of 
\eqref{eq1} equal $(a*b*c)\cup \{0\}$. This proves clam \eqref{lemzore.1}.
Claims \eqref{lemzore.2} and \eqref{lemzore.3} follow from the definition of $\bullet$.
\end{proof}

By the above one can consider multisemigroups {\em without} zero elements and understand that the
role of the zero element is played by the ``undefined'' multiplication, that is the case $a*b=\varnothing$.
This, in particular, unifies the notions of ``simple'' and ``$0$-simple'' multisemigroups
and semigroups.

A multisemigroup $(S,*)$ will be called a {\em quasi-semigroup} provided that for any $a,b\in S$ the product
$a*b$ is either empty or an element of $S$. Quasi-semigroups can be 
identified with semigroups with zero elements: Given a semigroup with a zero element we can take this zero 
element away and redefine the product to be empty whenever it was zero to obtain a quasi-semigroup. Conversely, 
let $(S,*)$ be a quasi-semigroup without a zero element. Consider the set $\displaystyle 
S^0:=S\cup\{0\}$, where we assume $0\not\in S$, and for $a,b\in S^0$ define
\begin{displaymath}
a\odot b:=\begin{cases}a*b;& a,b\in S,\, a*b\neq \varnothing;\\
\{0\},& \text{otherwise}.\end{cases} 
\end{displaymath}
Then $(S^0,\odot)$ becomes a semigroup with a zero element. Note that the previous construction $(S^0,\bullet)$
produces in this case a multisemigroup, not a semigroup.

\subsection{Homomorphisms}\label{s4.4}

Let $(S,*)$ and $(T,\bullet)$ be multisemigroups. A {\em strong homomorphism} from $S$ to $T$ is a map
$\varphi:S\to T$ such that for any $a,b\in S$ we have 
\begin{displaymath}%\label{eq6}
\bigcup_{s\in a*b}\{\varphi(s)\}=\varphi(a)\bullet\varphi(b).
\end{displaymath}
If $\varphi:S\to T$ is a strong homomorphism, it extends uniquely to a quantale homomorphism 
$\overline{\varphi}:2^S\to 2^T$ by setting $\overline{\varphi}(A):=\cup_{a\in A}\{\varphi(a)\}$ for $A\subset S$.
By definition, $\overline{\varphi}$ maps atoms of $2^S$ to atoms of $2^T$. A homomorphism of atomic quantales
which maps atoms to atoms is called {\em atomic}. Conversely, any atomic homomorphism from $2^S$ to $2^T$ gives, via
restriction to atoms, a strong homomorphism from $S$ to $T$. 

Let $\mathbf{MSemi}$ denote the category of multisemigroups with strong homomorphisms. By the above, this category is
equivalent to the category $\mathbf{QCABA}'$ of quantale structures on complete atomic Boolean algebras with atomic homomorphisms. 

As multisemigroups are defined in a multi-setting anyway, it is natural to extend the above as follows:
Let $(S,*)$ and $(T,\bullet)$ be multisemigroups. A {\em weak homomorphism} from $S$ to $T$ is a map
$\varphi:S\to 2^T$ such that for any $a,b\in S$ we have 
\begin{displaymath}%\label{eq7}
\bigcup_{s\in a*b}\varphi(s)=\varphi(a)\bullet\varphi(b).
\end{displaymath}
If $\varphi:S\to 2^T$ is a weak homomorphism, it extends uniquely to a quantale homomorphism 
$\overline{\varphi}:2^S\to 2^T$ by setting $\overline{\varphi}(A):=\cup_{a\in A}\varphi(a)$ for $A\subset S$.
Conversely, any quantale homomorphism from $2^S$ to $2^T$ gives, via restriction to atoms of $S$, a weak 
homomorphism from $S$ to $T$. 

Let $\overline{\mathbf{MSemi}}$ denote the category multisemigroups with weak homomorphisms. By the above, this 
category is equivalent to the category $\mathbf{QCABA}$ of quantale structures on complete atomic Boolean 
algebras with usual homomorphisms. 

\subsection{Submultisemigroups}\label{s4.5}

A non-empty subset $T$ of a multisemigroup $(S,*)$ is called a {\em submultisemigroup} provided that
$a*b\subset T$ for all $a,b\in T$. Any submultisemigroup $T$ is a multisemigroup with respect to the restriction
of $*$ to $T$.

Let $(S,*)$ be a multisemigroup and $X\subset S$ a non-empty subset. Then the  {\em submultisemigroup} 
$\langle X\rangle$ of $S$ {\em generated by $X$} is the minimal (with respect to inclusion) submultisemigroup of
$S$ containing $X$. Alternatively, $\langle X\rangle$ is the intersection of all submultisemigroups of 
$S$ containing $X$.

Let $(S,*)$ be a multisemigroup and $J$ a $\mathcal{J}$-class of $S$. Set $T=\langle J\rangle$. 
As $J$ is a $\mathcal{J}$-class, $T\setminus J$ is an ideal of $T$ and hence we can
consider the corresponding Rees quotient $T\setminus(T\setminus J)$. Note that $T\setminus(T\setminus J)$
might not be a semigroup even if $S$ is a semigroup. This construction is a natural multisemigroup analogue of 
the standard construction of the ``semigroup associated with a $\mathcal{J}$-class''.

\subsection{Congruences and quotients}\label{s4.6}

Let $(S,*)$ be a multisemigroup. An equivalence relation $\sim$ on $S$ is called a {\em left congruence}
provided that the following condition is satisfied: for any $a,b,c\in S$ such that $a\sim b$ and for any
$s\in c*a$ and $t\in c*b$ there are $s'\in c*b$ and $t'\in c*a$ such that $s\sim s'$ and $t\sim t'$.
A {\em right congruence} is defined similarly (using $a*c$ and $b*c$ at appropriate places) and a
{\em congruence} is an equivalence relation which is both a left and a right congruence (confer \cite{Da}).

Let $\sim$ be a congruence on $(S,*)$. Then the usual argument shows that the quotient set $S/\hspace{-1.5mm}\sim$ has the 
natural structure of a multisemigroup with respect to the induced multioperation $\circ$ defined for 
$A,B\in S/\hspace{-1.5mm}\sim$ as follows:
\begin{displaymath}
A\circ B=\{C\in S/\hspace{-1.5mm}\sim\,\vert\, \text{ there exist } a\in A,\,b\in B,\,c\in C\text{ such that }c\in a*b\}. 
\end{displaymath} 
Mapping $s\in S$ to its $\sim$-class $\overline{s}$ is obviously a strong homomorphism from the
multisemigroup $(S,*)$ to the multisemigroup $(S/\hspace{-1.5mm}\sim,\circ)$. We have
$A\circ B=\{\overline{s}\,\vert\, s\in A*B\}$ for $A,B\in S/\hspace{-1.5mm}\sim$, furthermore, 
for any $a,b\in S$ we have $\overline{a}\circ \overline{b}=\{\overline{s}\,\vert\,s\in a*b\}$.

Let $(T,\cdot)$ be another multisemigroup and $\varphi:S\to T$ a strong homomorphism. Then the equivalence
relation $\mathrm{Ker}(\varphi)$ on $S$ is easily seen to be a congruence. Hence, as usual, congruences on
multisemigroups are exactly kernels of strong homomorphisms.

Let $\sim$ be a congruence on a multisemigroup $(S,*)$. Then the map $j$ which sends $a\in S$ to its $\sim$-class 
extends to the map $j:2^S\to 2^S$, by $j(A)=\cup_{a\in A} j(a)$. It is straightforward to verify that this map 
is monotone, extensive, idempotent and satisfies $j(A)*j(B)\leq j(A*B)$. It follows that $j$ is a quantic 
nucleus, see \cite[Chapter 3]{Ro}. The quotient quantale defined by $j$ is a complete atomic Boolean algebra with 
the atoms $j(a)$, $a\in S$. Setting $T=S/\hspace{-1mm}\sim$, it is easy to see that this quotient quantale is just 
the quantale $2^T$ associated to the multisemigroup $T$. Conversely, assume that we are given a complete atomic 
Boolean algebra $2^X$ with a quantale structure on it and assume that $j:2^X\to 2^X$ is a nucleus such that 
the quotient quantale is a complete atomic Boolean algebra $2^Y$. On the multisemigroup $X$ we define a 
relation $\sim$ as follows: $x\sim y$ if and only if $j(\{x\})=j(\{y\})$. 
This relation is a congruence and the quotient multisemigroup is $Y$, 
the multisemigroup of atoms of $2^Y$. It follows that congruences on multisemigroups are in a bijective 
correspondence with congruences on quantale structures on complete atomic Boolean algebras.

\subsection{Representations of multisemigroups by binary relations}\label{s4.7}

Regular representations of multisemigroups by binary relations were constructed in \cite{Ea} in the
following way: Let $(S,*)$ be a multisemigroup. Then $*$ can be viewed as a binary relation  $S\times S\to S$, 
in particular, every $a\in S$ defines a binary relation $\tau_{a}$ on $S$ via $y\,\tau_a\, x$ if and only 
if $y\in a*x$ (this corresponds to the usual convention that maps operate from the right to the left).
We identify binary relations with square matrices over $\mathbb{B}$, whose rows and columns
are indexed by elements of $S$ (i.e. $a\,\tau\,b$ if and only if the intersection of row $a$ and  
column $b$ of the matrix of $\tau$ contains $1$). Then for the usual sum $+$ and product $\circ$ of binary 
relations (i.e. sum and product of Boolean matrices) we have 
\begin{displaymath}
\tau_a\circ \tau_b=\sum_{s\in a*b}\tau_s.
\end{displaymath}
This is the left regular representation of $(S,*)$. The right regular representation is defined similarly.

This representation extends to a homomorphism from the quantale $2^{S}$ to the
quantale $\mathbf{B}(S)$ of binary relations on $S$. Therefore it is natural to generalize this and call a 
{\em representation} of $(S,*)$ (by binary relations) any quantale homomorphism from $2^{S}$ 
to any quantale of binary relations.

Similarly, given an $\mathcal{L}$-class $L$ in $S$, for $a \in S$ we define the binary relation $\lambda_a$ on $L$ via
$y\,\lambda_a\, x$ if and only if $y\in a*x$. This extends, using addition of binary relations, 
to a homomorphism from the quantale 
$2^{S}$ to the quantale $\mathbf{B}(L)$ of binary relations on $L$. This is the
representation of $S$ associated to $L$ (confer \cite[Section~10]{GM2}).

\section{Strongly simple multisemigroups}\label{s5}

\subsection{Simple and strongly simple multisemigroups}\label{s5.1}

A multisemigroup $(S,*)$ is called {\em simple} if for any $a\in S$ we have $S^1*a*S^1=S$, that is $S$ has a unique
$\mathcal{J}$-class (confer \cite{JSM}). From now on, unless stated otherwise, we assume that $S$ does not contain 
a zero element (see Subsection~\ref{s4.2}). This excludes one special case: when $(S,*)$ is a singleton group. 
In the latter case $(S,*)$ contains both the identity and the zero elements and they coincide. This case is 
usually excluded in  the classical ring theory as well.

Recall the example of a two-element multisemigroup, $(S,*)$, from Subsection~\ref{s3.2}. This multisemigroup is finite simple 
with a unique $\mathcal{R}$-class but two different $\mathcal{L}$-classes which are, moreover, comparable with respect 
to $\leq_L$. This shows a significant difference between multisemigroups and semigroups as no analogous situation
is possible for semigroups.

From now on, if the converse is not explicitly stated, by a {\em minimal} left ideal of $S$ we mean a minimal non-empty left ideal, that is a non-empty left ideal $I$ of $S$ such that for any left ideal $J$ of $S$ the inclusion $J\subseteq I$ implies that $J=I$ or $J=\varnothing$. In a similar way we will also use the notions of a minimal right ideal and a minimal (two-sided) ideal.

If $(S,*)$ is a multisemigroup, an element $s\in S$ will be called a {\em quark} provided that $S^1*s$ is a minimal
left ideal and $s*S^1$ is a minimal right ideal. For any quark $s$ we thus have $\mathcal{L}_s=S^1*s$
and $\mathcal{R}_s=s*S^1$. We denote by $Q(S)$ the set of all quarks in $(S,*)$.
The set $Q(S)$ will be called the {\em support} of $S$. A simple multisemigroup $(S,*)$ will be called 
{\em strongly simple} if $S=Q(S)$. For instance, any completely simple semigroup is a strongly simple multisemigroup.
Furthermore, the quasi-semigroup associated to a completely $0$-simple semigroup as described at the end of
 Subsection~\ref{s4.2} is a strongly simple multisemigroup.
A special case of the last example is the singleton multisemigroup $\mathbf{0}:=(\{0\},*)$,
where $0*0=\varnothing$. 
Moreover, the multisemigroup $\mathbf{0}$ is very special because of the following:

\begin{proposition}\label{prop29}
Let $(S,*)$ be a multisemigroup. If $(S,*)$ contains only one $\mathcal{H}$-class, then either
$S\cong \mathbf{0}$ or $S$ is a hypergroup.
\end{proposition}

\begin{proof}
Assume that $S=\mathcal{H}_a$ for every $a\in S$. If $|S|=1$ then $S\cong \mathbf{0}$
(note that $S$ cannot be isomorphic to the trivial semigroup as the trivial semigroup contains a zero element) 
and hence is a hypergroup. Suppose that $|S|>1$. Let $a\in S$. 
We have $a*S^1=b*S^1\ni b$, for any $b\in S$,  which implies that $a*S\supset 
S\setminus\{a\}$. Assume that $a*S=S\setminus\{a\}$. Then
$S\setminus\{a\}$ is a proper right ideal of $S$ and hence for any $b\in S\setminus\{a\}$ we have
$b*S^1\subset S\setminus\{a\}$, a contradiction. Hence
$a*S=S$. Similarly one shows that $S*a=S$ and the claim follows. 
\end{proof}

For $Q(S)$ we have the following elementary property:

\begin{proposition}\label{prop7}
Let $(S,*)$ be a multisemigroup with non-empty support. Then we have:
\begin{enumerate}[$($a$)$]
\item\label{prop7.1} $Q(S)$ and every non-empty intersection of $Q(S)$ with a $\mathcal{J}$-class of $S$
is a submultisemigroup.
\item\label{prop7.2} The multisemigroup $Q(S)$ is a disconnected union (as defined in Subsection \ref{s3.13})
of its non-empty intersections with $\mathcal{J}$-classes of $S$.
\end{enumerate}
\end{proposition}

\begin{proof}
Let $a,b\in Q(S)$. Every element $s\in a*b$ belongs to $S^1*b\cap a*S^1$ and 
hence to $Q(S)$. This shows that $Q(S)$ is a submultisemigroup and claim \eqref{prop7.1} follows.
 
Let $a,b\in Q(S)$ be such that $a$ and $b$ belong to different $\mathcal{J}$-classes of $S$.
If $s\in a*b$, then $s\in S^1*b$ and hence $s\,\mathcal{L}\,b$ by the minimality of $S^1*b$. At the same 
time $s\in a*S^1$ and hence $s\,\mathcal{R}\,a$ by the minimality of $a*S^1$. This implies $s\,\mathcal{J}\,b$ and
$s\,\mathcal{J}\,a$ which means that $a\,\mathcal{J}\,b$ contradicting our assumptions. This implies that 
$a*b=\varnothing$ and hence $Q(S)$ is the disconnected union of  its intersections with  $\mathcal{J}$-classes of $S$.
\end{proof}

%\subsection{Simple multisemigroups}\label{s5.105}

\begin{lemma}\label{lem22}
Let $(S,*)$ be a simple multisemigroup, $I$ a non-empty left ideal of $S$ and $J$ a non-empty right ideal of $S$.
Then $I\cap J\neq \varnothing$ and $J*I\neq \varnothing$.
\end{lemma}

\begin{proof}
The claim is obvious in the case $|S|=1$, so we assume that $|S|>1$. Since $I\cap J\supset J*I$, it is enough to show
that $J*I\neq \varnothing$. Assume that this is not the case, that is $J*I=\varnothing$. 
Since $S^1*J$ is a non-empty ideal of $S$, we 
have  $\displaystyle S=S^1*J$ since $S$ is simple. Similarly 
$S=I*S^1$. But then
\begin{displaymath}
S*S=S^1*J*I*S^1=\varnothing.
\end{displaymath}
This is, however, not possible if $S$ is simple and $|S|>1$. This completes the proof.
\end{proof}

We say that a multisemigroup $(S,*)$ is {\em of finite type} provided that every (non-empty) left ideal of $S$ contains a 
minimal left ideal and every (non-empty) right ideal of $S$ contains a minimal right ideal. Clearly, every finite
multisemigroup is of finite type. 

\begin{corollary} Let $(S,*)$ be a simple multisemigroup of finite type. Then $Q(S)\neq \varnothing$.
\end{corollary}

\begin{lemma}\label{lem23}
Let $(S,*)$ be a multisemigroup and $a,b\in S$. Then the following statements are equivalent:
\begin{enumerate}[$($a$)$]
\item\label{lem23.1} $a*b=\varnothing$.
\item\label{lem23.2} $(S^1*a)*(b*S^1)=\varnothing$.
\item\label{lem23.3} $\mathcal{L}_a*\mathcal{R}_b=\varnothing$. 
\item\label{lem23.4} There is $s\in \mathcal{L}_a$ and
$t\in \mathcal{R}_b$ such that $s*t=\varnothing$. 
\end{enumerate}
\end{lemma}

\begin{proof}
That \eqref{lem23.2} implies \eqref{lem23.3} is obvious.
That \eqref{lem23.3} implies \eqref{lem23.4} is obvious.
That \eqref{lem23.1} implies \eqref{lem23.2} follows from 
the following  computation, which uses associativity: $(S^1*a)*(b*S^1)= S^1*(a*b)*S^1=\varnothing$.
The same argument shows that \eqref{lem23.4} implies $(S^1*s)*(t*S^1)=\varnothing$
and, as $a\in S^1*s$ and $b\in t*S^1$, we get that \eqref{lem23.4} implies \eqref{lem23.1}.
\end{proof}

\begin{corollary}\label{cor24}
Let $(S,*)$ be a multisemigroup and $H$ an $\mathcal{H}$-class in $S$. Then the following conditions are 
equivalent:
\begin{enumerate}[$($i$)$]
\item\label{cor24.1} $H*H\neq \varnothing$.
\item\label{cor24.2} There exist $s,t\in H$ such that $s*t\neq\varnothing$.
\item\label{cor24.3} For all $s,t\in H$ we have $s*t\neq\varnothing$.
\end{enumerate}
\end{corollary}

\begin{proof} 
Obviously, condition \eqref{cor24.3} implies condition \eqref{cor24.1} and  condition \eqref{cor24.1} implies 
condition \eqref{cor24.2}. For $a\in H$ we have $H=\mathcal{L}_a\cap\mathcal{R}_a$ and the fact that 
condition \eqref{cor24.2} implies condition \eqref{cor24.3} follows from Lemma~\ref{lem23}.
\end{proof}

\begin{proposition}\label{prop25}
Let $(S,*)$ be a multisemigroup and $a,b\in Q(S)$. Then $a*b\subset \mathcal{L}_b
\cap \mathcal{R}_a$. Moreover, if $S$ is simple and $H:=\mathcal{L}_a\cap \mathcal{R}_b$, 
then $a*b\neq\varnothing$ if and only if $H*H\neq \varnothing$.
\end{proposition}

\begin{proof} 
The first claim is obvious, so we prove the second one. Suppose $a*b\neq\varnothing$. Since $a,b\in Q(S)$ it follows that $\mathcal{L}_a$, $\mathcal{L}_b$ are minimal left ideals and $\mathcal{R}_a$, $\mathcal{R}_b$ are minimal right ideals. By Lemma~\ref{lem22} we have $H\neq \varnothing$  and hence there exists $y\in H$. 
We have $y*y\neq\varnothing$ by Lemma~\ref{lem23} and thus $H*H\neq \varnothing$.
Conversely, if  $H*H\neq \varnothing$, then $a*b\neq \varnothing$ by Lemma~\ref{lem23}.
\end{proof}

\begin{proposition}\label{cor28}
Let $(S,*)$ be a multisemigroup such that $Q(S)\neq\varnothing$. Then $Q(S)$ is a union of $\mathcal{H}$-classes of $S$, the restriction $\sim$ of the relation $\mathcal{H}$ 
to $Q(S)$ is a congruence on $Q(S)$ and $Q(S)/\hspace{-1.5mm}\sim$ is a quasi-semigroup with singleton $\mathcal{H}$-classes.
\end{proposition}

\begin{proof}
If $a\in Q(S)$ and $b\,\mathcal{H}\, a$, then clearly $b\in Q(S)$, so $Q(S)$ is a union of $\mathcal{H}$-classes.  Let $a,b\in Q(S)$ be such that $a\sim b$. Then there exist a minimal left ideal $I$ and a minimal right ideal $J$
of $S$ such that $a,b\in I\cap J$. Let further $c\in Q(S)$. Then there exist a minimal left ideal $I'$ and a minimal right ideal $J'$
of $S$ such that $c\in I'\cap J'$. But then we have both $a*c\subset I'\cap J$ and $b*c\subset I'\cap J$, which
implies that $\sim$ is a right congruence on $Q(S)$. That $\sim$ is a left congruence is proved similarly. Hence $\sim$ is a congruence
on $Q(S)$. The elements of $Q(S)/\hspace{-1.5mm}\sim$ can be identified with pairs $(I,J)$, where $I$ is a minimal left ideal in $S$
and $J$ is a minimal right ideal in $S$ such that $I\cap J \neq\varnothing$ (the pair $(I,J)$ corresponds to $I\cap J$). The above also shows that 
the multiplication in $Q(S)/\hspace{-1.5mm}\sim$ is given by 
\begin{displaymath}
(I,J)\cdot (I',J')=
\begin{cases} 
(I',J), & (I\cap J)*(I'\cap J')\neq \varnothing;\\
\varnothing, & \text{ otherwise}.
\end{cases}
\end{displaymath}
This is obviously a  quasi-semigroup with singleton $\mathcal{H}$-classes. 
\end{proof}

\subsection{Structure of strongly simple multisemigroups}\label{s5.2}

Let $(S,*)$ be a strongly simple multisemigroup.  Then $S$ is both, 
the (disjoint) union of its minimal left ideals and the (disjoint) union of its minimal right ideals. At the same time, if $S$ is simple and every element of $S$ generates a minimal left and a minimal right ideal, then
$S$ is strongly simple. Furthermore, 
$S^1*a=\mathcal{L}_a$ and $a*S^1=\mathcal{R}_a$ for any $a\in S$. In this subsection we show how the results of the 
previous subsection can be strengthened in the case of strongly simple multisemigroups. It turns out that hypergroups arise
naturally and play a very important role in this case (similar to the role of group $\mathcal{H}$-classes
for completely $0$-simple semigroups).

\begin{theorem}[Structure of strongly simple multisemigroups]
\label{thm31}
Let $(S,*)$ be a strongly simple multisemigroup.
\begin{enumerate}[$($a$)$]
\item\label{thm31.0} For any $a,b\in S$ we have $\mathcal{L}_a\cap \mathcal{R}_b\neq\varnothing$.
\item\label{thm31.1} If $H$ is an $\mathcal{H}$-class in $S$, then either $H*H=\varnothing$ or $H$ is a hypergroup.
\item\label{thm31.2} For $a,b\in S$ we have $a*b\neq\varnothing$ if and only if $\mathcal{L}_a\cap \mathcal{R}_b$
is a hypergroup.
\item\label{thm31.3} Assume $S\not\cong \mathbf{0}$. Then every $\mathcal{L}$-class and every $\mathcal{R}$-class 
in $S$ contains at least one hypergroup $\mathcal{H}$-class.
\item\label{thm31.4} Let $a,b\in S$ be such that $a\,\mathcal{R}\,b$ and let $s\in S^1$ be such that 
$b\in a*s$. Then the multivalued map $x\mapsto x*s$ is surjective from $\mathcal{L}_a$ to $\mathcal{L}_b$ and 
preserves both $\mathcal{R}$- and $\mathcal{H}$-classes, that is, for any $\mathcal{H}$-class $H\subset \mathcal{L}_a$
we have $H*s$ is an $\mathcal{H}$-class in $\mathcal{L}_b$ and the two $\mathcal{H}$-classes $H$ and $H*s$ are contained
in the same $\mathcal{R}$-class.
\item\label{thm31.5} Assume $S\not\cong \mathbf{0}$.  Let $I$ be a minimal left ideal of $S$ and $J$ a minimal 
right ideal of $S$. Then $I\cap J=J*I$.
\item\label{thm31.6} $\mathcal{H}$ is a congruence on $S$ and the quasi-semigroup $S/\mathcal{H}$ is bisimple 
(i.e. contains a unique $\mathcal{D}$-class) with singleton $\mathcal{H}$-classes. 
\end{enumerate}
\end{theorem}

\begin{proof}
Claim \eqref{thm31.0} is a special case of Lemma~\ref{lem22}.

Let us prove claim \eqref{thm31.1}. Let $H$ be an $\mathcal{H}$-class in $S$ and assume that $H*H\neq \varnothing$.  Let $L$ and $R$ be the $\mathcal{L}$-class and the $\mathcal{R}$-class whose intersection is $H$.
We have that $L*R$ is an ideal of $S$ and is non-empty (since $H*H\neq \varnothing$). Hence $L*R=S$. We first show that $H*H=H$. Let $a\in H$. Since $L*R=S\supseteq H$, it follows that $a\in b*c$ for some $b\in L$ and $c\in R$. But $b*c\subseteq b*S^1\cap S^1*c$. Since the latter is an $\mathcal{H}$-class and $b*c\cap H\neq\varnothing$ then $H=b*S^1\cap S^1*c$. Since also $H=S^1*b\cap c*S^1$ it follows that  $b\in H$ and $c\in H$. Therefore $H=H*H$, as required.
Let $a\in H$. Show that $H*a=H$. Let $c\in H$. As $H=H*H$, there are some $x,y\in H$ such that $c\in x*y$. 
Now $a\,{\mathcal R}\,x$ implies $x\in a*t$ for some $t\in S^1$.
We then have $c\in x*y\subseteq  a*t*y$.  So $c\in a*L$ (because $t* y\subseteq L$). We have shown that $H\subseteq a*L$.  So $H=a*L$ since clearly $a*L\subseteq H$. Observe that $L*a$ is a non-empty left ideal and so $L*a=L$. Hence we have
$$H*a=(a*L)*a =a*(L*a)= a*L=H. 
$$
Similarly we verify that $a*H=H$ and claim \eqref{thm31.1} follows. 

Claim \eqref{thm31.2} follows from
claim \eqref{thm31.1} and Proposition~\ref{prop25}. 

We prove claim \eqref{thm31.3} for $\mathcal{L}$-classes (for $\mathcal{R}$-classes the proof is similar).
If $|S|=1$, the claim is obvious. Consider the case  $|S|>1$. Then for any different $a,b\in S$ we have $S^1*a*S^1=S^1*b*S^1\ni b$,
which implies $S*S\neq\varnothing$, in particular, there exist $s,t\in S$ such that  $s*t\neq \varnothing$. 
Set $H:=\mathcal{L}_s\cap \mathcal{R}_t$. Claim \eqref{thm31.1} implies that $H$ is a hypergroup
$\mathcal{H}$-class in $\mathcal{L}_s$. Let $L$ be an $\mathcal{L}$-class
different from $\mathcal{L}_s$, $x\in L\cap \mathcal{R}_t$ and $y\in H$. Then $y\neq x$ and $y\,\mathcal{R}\,x$ 
and hence there is $u\in S$ such that $y\in x*u$. Consider $H'=S^1*x\cap u*S^1$, which is an
$\mathcal{H}$-class by Lemma~\ref{lem22}. Proposition~\ref{prop25} along with claim \eqref{thm31.1}
imply that $H'$ is a hypergroup. Claim \eqref{thm31.3} follows.

In the setup of claim \eqref{thm31.4} we have $\mathcal{L}_a*s\neq\varnothing$ and hence $\mathcal{L}_a*s$ 
is a left ideal of
$S$ contained in $\mathcal{L}_s=\mathcal{L}_b$. Then $\mathcal{L}_a*s=\mathcal{L}_b$ by minimality of
$\mathcal{L}_b$. If $c\in \mathcal{L}_b$ is such that $c\in x*s$, then $c\,\mathcal{R}\,x$ and hence the map
preserves $\mathcal{R}$-classes and, consequently, $\mathcal{H}$-classes. This proves claim \eqref{thm31.4}.

To prove claim \eqref{thm31.5} we note that $J*I\neq \varnothing$ by Lemma~\ref{lem22}
and $J*I\subset I\cap J$. Let $a\in J$ and $b\in I$ be such that $a*b\neq\varnothing$.
Then claim \eqref{thm31.4} implies that $\mathcal{H}_a*b=I\cap J$. This proves claim \eqref{thm31.5}.

Finally, claim \eqref{thm31.6} follows from Proposition~\ref{cor28} and claims
\eqref{thm31.2} and \eqref{thm31.3}.
\end{proof}

\subsection{Example: a finite simple multisemigroup with nontrivial inclusions of principal one-sided ideals}\label{s5.205}

Let $(X,\leq)$ be a partially ordered set possessing a unique minimal element $y$. Consider the rectangular band $T=X\times X$. For $a\in X$ let $a^{\downarrow}=\{x\in X:x\leq a\}$ be the {\em downward closure} of $a$. We define the function $f:T\to 2^T$ by setting
\begin{equation}\label{eq111}
 f(a,b)=a^{\downarrow}\times b^{\downarrow}.\end{equation}
 It is easy to check that the function $f$ satisfies the equality $f(f(x)f(y))=f(x)f(y)$ and therefore
$S=(T,\ast)$ with $(a,b)\ast (c,d)=f(a,b)f(c,d)$ is a multisemigroup (see Subsection~\ref{s3.14}).
Moreover,
one checks that $S$ is a simple multisemigroup and $Q(S)=\{(y,y)\}$. Clearly $Q(S)$ is a strongly simple multisemigroup. 
Note that $(a,b)\ast S^1$ is a minimal right ideal if and only if $a=y$; $S^1\ast (a,b)$ is a minimal 
left ideal if and only if  $b=y$. Let $a_1<\dots <a_m$ be a chain in $X$ and $c\in X$ be any element.  Then we have a chain of principal left ideals 
\begin{equation}\label{eq112}
S^1\ast (c,a_1)\subsetneq S^1\ast (c,a_2)\subsetneq \dots \subsetneq S^1\ast (c,a_m). 
\end{equation}

The above example can be generalized as follows. Let $I$ be a set and let $(X_i,\leq_i)$, $i\in I$, be partially ordered sets each possessing a unique minimal element $y_i$, $i\in I$. We assume that the sets $X_i$ are pairwise disjoint. Let $Y=\{y_i\colon i\in I\}$. Let $X=\cup_{i\in I}X_i$. We define on $X$ the partial order $\leq$ by setting $x\leq y$ provided that there is $i$ such that $x,y\in X_i$ and $x\leq_i y$. Consider the rectangular band $T=X\times X$
and define $f:T\to 2^T$ as given in \eqref{eq111}.

The map $f$ gives a simple multisemigroup $S=(T,*)$ according to Subsection~\ref{s3.14}. The
submultisemigroup $Q(S)$ of $S$ is equal to $Y\times Y$ and is a rectangular band, hence is a strongly simple multisemigroup (in fact, a semigroup); 
$(a,b)\ast S^1$ is a minimal right ideal if and only if $a\in Y$; $S^1\ast (a,b)$ is a minimal 
left ideal if and only if  $b\in Y$. Let   
$a_1<\dots <a_m$ be a chain in $X$. Let $c\in Z$ be any element.  Then  \eqref{eq112} gives a chain of left ideals.

\subsection{Example: a finite simple multisemigroup whose support is not simple}\label{s5.207}

Let $X=\{1,2\}$ with the order given by $1<2$. Let $S=X\times X$. Let $(S,\ast)=S(X)$ be the multisemigroup constructed in Subsection \ref{s5.205}. That is,  the multiplication $\ast$ on $S$ is given by the map $\ast:S\times S\to 2^S$:\begin{displaymath}
(i,j)\ast (k,l)= i^{\downarrow}\times l^{\downarrow}.
\end{displaymath}
Let $T=S\cup (1',1')$ (here $1'$ is an element different from both $1$ and $2$). Define the map $\pi:T\to S$ by setting $\pi(a)=a$, if $a\in S$, and $\pi((1',1'))=(1,1)$. Let  $\circ:T\times T\to 2^T$ 
be defined as follows:
\begin{displaymath}
x\circ y = 
\begin{cases}
\pi(x)\ast \pi(y),&  x\in \{(1,1), (1',1'),(1,2)\} \text{ and }  y\in \{(1,1), (1',1'),(2,1)\};\\
(\pi(x)\ast \pi(y))\cup (1',1'),& \text{ otherwise }. 
\end{cases}
\end{displaymath}

\begin{lemma} \label{lem31}
$(T,\circ)$ is a multisemigroup.
\end{lemma}

\begin{proof} 
Observe that we have $(1,1)\circ x = (1',1')\circ x$ and $x\circ (1,1)=x\circ (1',1')$ for any $x\in T$. 

Let $x,y,z\in T$. We have to verify that $(x\circ y)\circ z = x\circ (y\circ z)$. Note that $(1,1)\in s\circ t$ for any $s,t\in T$. It follows that $(x\circ y)\circ z=(\pi(x)\ast \pi(y))\circ z$ and, similarly, $x\circ (y\circ z)=x\circ (\pi(y)\ast \pi(z))$.

By the definition, $(\pi(x)\ast \pi(y))\circ z$ equals either $(\pi(x)\ast \pi(y))\ast \pi(z)$ or $((\pi(x)\ast \pi(y))\ast \pi(z))\cup \{(1',1')\}$. Let us establish when it equals $(\pi(x)\ast \pi(y))\ast \pi(z)$.
From the definition we have that this happens if and only if $z\in \{(1,1), (1',1'),(2,1)\}$ and  
$\pi(x)\ast \pi(y)\subseteq \{(1,1),(1,2)\}$. The latter inclusion holds if and only if 
$x\in \{(1,1), (1',1'),(1,2)\}$. 

Similarly, $x\circ (\pi(y)\ast \pi(z))$ equals either $\pi(x)\ast (\pi(y)\ast \pi(z))$ or $(\pi(x)\ast (\pi(y)\ast \pi(z)))\cup \{(1',1')\}$ and one shows that $x\circ (\pi(y)\ast \pi(z))=\pi(x)\ast (\pi(y)\ast \pi(z))$ if and 
only if $x\in \{(1,1), (1',1'),(1,2)\}$ and $z\in \{(1,1), (1',1'),(2,1)\}$. The equality 
$(x\circ y)\circ z = x\circ (y\circ z)$ follows.
\end{proof}

It is easy to check that $T$ is a simple multisemigroup.  For $x\in T$ we have 
$x\circ T^1= (\pi(x)\ast S^1)\cup \{(1',1')\}$, $T^1\circ x = (S^1\ast \pi(x)) \cup \{(1',1')\}$.
Furthermore, $Q(T)=\{(1,1), (1',1')\}$, $(1,1)\circ T^1 \cap T^1\circ (1,1)=\{(1,1),(1',1')\}$.
Both $(1,1)$ and $(1',1')$ belong to the same $\mathcal{H}$-class of $T$.
At the same, $Q(T)\circ Q(T)=\{(1,1)\}$ and hence $Q(T)$ is not a hypergroup.
Clearly, $(1,1)$ and $(1',1')$ are not in the same $\mathcal{J}$-class of $Q(T)$ and so $Q(T)$ is not simple.

\subsection{Example: a finite simple multisemigroup for which ${\mathcal R}\circ {\mathcal L}\neq {\mathcal  L}\circ {\mathcal  R}$}\label{s5.31} Let $X=\{1,2\}$ and $S=X\times X$. Let $(S,\ast)=S(X)$ be the multisemigroup from Subsection  \ref{s5.207}.  We have 
$(i,j)\ast (k,l)\ast (p,q)= i^{\downarrow}\times q^{\downarrow}$.

Let $T=S\setminus \{(2,2)\}$.
We define a map $T\times T\to 2^T$, $(x,y)\mapsto x\bullet y$, by
$$
(i,j)\bullet (k,l) = ((i,j)\ast (k,l))\setminus \{(2,2)\}.
$$

\begin{lemma} $(T,\bullet)$ is a multisemigroup. 
\end{lemma}
\begin{proof} 
Let $(i,j)$, $(k,l)$,  $(p,q)\in T$.
Observe that if $(i,j)\ast (k,l)\ni (2,2)$, then $(i,j)\ast (k,l)\ni (2,1)$ and hence $((i,j)\ast (k,l))\ast (p,q)=(((i,j)\ast (k,l))\setminus \{(2,2)\})\ast (p,q)$.
Using this we calculate 
\begin{multline*}
((i,j)\bullet (k,l))\bullet (p,q)=((((i,j)\ast (k,l))\setminus \{(2,2)\})\ast (p,q))\setminus \{(2,2)\}=\\=
(((i,j)\ast (k,l))\ast (p,q))\setminus \{(2,2)\}. 
\end{multline*}
Similarly we show that 
$$(i,j)\bullet ((k,l)\bullet (p,q))=((i,j)\ast ((k,l)\ast (p,q)))\setminus \{(2,2)\}.
$$
Associativity of $\bullet$ follows.
\end{proof}

It is easy to see that $T$ is simple and $Q(T)=\{(1,1)\}$. Further,
it is easy to calculate Green's relations on $T$. We have ${\mathcal R}_{(1,1)}=\{(1,1),(1,2)\}$ and ${\mathcal R}_{(2,1)}=\{(2,1)\}$. Also ${\mathcal L}_{(1,1)}=\{(1,1),(2,1)\}$ and ${\mathcal L}_{(1,2)}=\{(1,2)\}$.
We see that $(1,2)\,\, {\mathcal R}\,\, (1,1)\,\, {\mathcal L}\,\, (2,1)$ so that $((1,2),(2,1))\in {\mathcal R}\circ {\mathcal L}$. But since  ${\mathcal L}_{(1,2)}\cap {\mathcal R}_{(2,1)}=\varnothing$, we have  $((1,2),(2,1))\not\in {\mathcal L}\circ {\mathcal R}$. In particular, ${\mathcal L}\circ {\mathcal R}\neq {\mathcal R}\circ {\mathcal L}$. Note that
we always have ${\mathcal L}\circ {\mathcal R}= {\mathcal R}\circ {\mathcal L}$ for semigroups.

\subsection{Simple  multisemigroups with identity}\label{s5.209}

A simple finite semigroup with identity is a group. For multisemigroups the situation is much more complicated.

\begin{lemma}\label{lem91}
Let $(S,*)$ be a strongly simple multisemigroup with identity. Then $S$ is a hypergroup.
\end{lemma}

\begin{proof}
Since $S$ is strongly simple, $S=S*1$ is a minimal left ideal and $S=1*S$ is a minimal right ideal. Hence
$S*a=S$ and $a*S=S$ for any $a\in S$ and thus $S$ is a hypergroup.
\end{proof}

One could expect that even a simple finite multisemigroup with identity should be a hypergroup.
Unfortunately, this is not the case. Indeed, consider $S=\{1,a,b,t\}$ with the multioperation $*$ defined 
by the following Cayley table:
\begin{displaymath}
\begin{array}{c||c|c|c|c}
*&1&a&b&t\\ \hline\hline
1&1&a&b&t\\ \hline
a&a&a&S&\{a,t\}\\ \hline
b&b&t&b&t\\ \hline
t&t&t&\{b,t\}&t
\end{array}
\end{displaymath}
It is straightforward to verify that $(S,*)$ is a multisemigroup. This multisemigroup is simple; it contains 
an identity element; it consists of idempotents; it has a unique minimal left ideal, namely $\{a,t\}$; it has a 
unique minimal right ideal, namely $\{b,t\}$; we have $Q(S)=\{t\}$ and $S$ is not a hypergroup.

Below we collect some properties of simple multisemigroups of finite type with identity:

\begin{proposition}\label{prop91}
Let $(S,*)$ be a simple multisemigroup of finite type with identity.
\begin{enumerate}[$($a$)$]
\item\label{prop91.1} If $I$ is a non-empty left ideal of $S$ and $J$ is a non-empty  right ideal of $S$, then
$I\cap\mathcal{R}_1\neq\varnothing$ and $J\cap\mathcal{L}_1\neq\varnothing$.
\item\label{prop91.2} If $I$ and $I'$ are two non-empty left ideals of $S$, then $I*I'\neq\varnothing$.
\item\label{prop91.3} If $J$ and $J'$ are two non-empty right ideals of $S$, then $J*J'\neq\varnothing$.
\end{enumerate}
\end{proposition}

\begin{proof}
Since $S$ contains the identity, it contains a unique (possibly empty) maximal proper right ideal (the union of all right ideals of $S$ which do not contain the identity), call it $\mathcal{K}$. Clearly, 
$\mathcal{R}_1=S\setminus \mathcal{K}$.

Let $I$ be a non-empty  left ideal of $S$. Then $X:=I*S$ is a non-empty two-sided ideal of $S$
since $S$ contains the identity
and hence $X=S$ since $S$ is simple. If $I\subset \mathcal{K}$, then $X\subset \mathcal{K}$ and hence $X$ 
cannot be equal to $S$, a contradiction. Therefore $I$ intersects $S\setminus \mathcal{K}=\mathcal{R}_1$ 
non-trivially. Similarly one shows that every non-empty  right ideal intersects $\mathcal{L}_1$ non-trivially, 
which proves claim \eqref{prop91.1}.

Let $I$ and $I'$ be two non-empty left ideals in of $S$. Then from the previous paragraph we have:
$(I*I')*S=I*(I'*S)=I*S=S$, thus $I*I'\neq\varnothing$. This proves claim \eqref{prop91.2}.
Claim \eqref{prop91.3} is proved similarly to claim \eqref{prop91.2}.
\end{proof}

\subsection{A Kazhdan-Lusztig example in type $B$}\label{s5.3}

In Subsection~\ref{s5.2} we established many elementary properties of strongly simple multisemigroups
that are similar to properties of  $0$-bisimple semigroups. One could observe that the multisemigroup version of
Green's lemma does not assert that $c\mapsto c*x$ is a {\em bijection} from $H$ to $H*x$.
This turns out to be false for multisemigroups in general. Here we give an explicit example.

Let $W=\{e,s,t,st,ts,sts,tst,stst\}$ be a Weyl group of type $B_2$ (the generators $s$ and $t$ satisfy
$s^2=t^2=e$ and $stst=tsts$) and $(W,*)$ be the corresponding 
Kazhdan-Lusztig multisemigroup as explained in Subsection~\ref{s3.9}. Denote by $T$ the Rees quotient
of the submultisemigroup $W\setminus\{e\}$ by the zero element $\{stst\}$. Then a direct computation shows
that $T$ is a strongly simple multisemigroup with the following Cayley table:
\begin{displaymath}
\begin{array}{c||c|c|c|c|c|c}
*&s&t&st&ts&sts&tst\\
\hline\hline 
s&s&st&st&\{s,sts\}&sts&st\\
\hline
t&ts&t&\{t,tst\}&ts&ts&tst\\
\hline
st&\{s,sts\}&st&st&\{s,sts\}&\{s,sts\}&st\\
\hline
ts&ts&\{t,tst\}&\{t,tst\}&ts&ts&\{t,tst\}\\
\hline
sts&sts&st&st&\{s,sts\}&s&st\\
\hline
tst&ts&tst&\{t,tst\}&ts&ts&t\\
\end{array}
\end{displaymath}
It follows that $T$ has the following egg-box diagram (in which all $\mathcal{H}$-classes are hypergroups):
\begin{displaymath}
\begin{array}{c||c|c}
&\mathcal{L}_s&\mathcal{L}_t\\
\hline\hline
\mathcal{R}_s&\{s,sts\}&\{st\}\\
\hline
\mathcal{R}_t&\{ts\}&\{t,tst\}.
\end{array}
\end{displaymath}

\section{The number of multisemigroups}\label{s6}

In this section we compare asymptotic properties of associativity for ordinary binary operations and
for multivalued binary operations.

\subsection{The number of semigroups}\label{s6.1}

Let $n$ be a positive integer and $N_n:=\{1,2,\dots,n\}$. A binary operation on $N_n$ corresponds to the choice
of a square $n\times n$ matrix with coefficients in $N_n$. Let $\mathrm{Mat}_{n\times n}(N_n)$ be the set of
all such matrices. Clearly, $|\mathrm{Mat}_{n\times n}(N_n)|=n^{n^2}$. Let $X$ denote the subset of 
$\mathrm{Mat}_{n\times n}(N_n)$ which consists of all matrices corresponding to {\em associative} binary
operation. By \cite[Equation~(3.6)]{KRS}, we have 
\begin{displaymath}
|X|=\left(\frac{n}{(2e+o(1))\ln n}\right)^{n^2}
\end{displaymath}
and hence, dividing by $n^{n^2}$, we obtain the following corollary, which says that ``almost all'' binary 
operations are not associative:

\begin{corollary}\label{cor51}
We have $\frac{|X|}{n^{n^2}}\to 0$ when $n\to \infty$. 
\end{corollary}

\subsection{The number of multisemigroups}\label{s6.2}

A multivalued binary operation on $N_n$ corresponds to the choice of a square $n\times n$ matrix with coefficients 
in $2^{N_n}$.  Let $\mathrm{Mat}_{n\times n}(2^{N_n})$ be the set of all such matrices. Clearly, 
$|\mathrm{Mat}_{n\times n}(2^{N_n})|=2^{n^3}$. Let $Y$ denote the subset of 
$\mathrm{Mat}_{n\times n}(2^{N_n})$ which consists of all matrices corresponding to {\em associative} 
multivalued binary operation (i.e. those satisfying \eqref{eq1}). Let $Y'$ denote the subset of 
$\mathrm{Mat}_{n\times n}(2^{N_n})$ which consists of all matrices corresponding to  
multivalued binary operation defining a hypergroup. The following claim, which says that 
``almost all'' multivalued binary operations are associative, is in striking contrast with Corollary~\ref{cor51}:

\begin{theorem}\label{thm52}
We have both $\frac{|Y|}{2^{n^3}}\to 1$ and $\frac{|Y'|}{2^{n^3}}\to 1$ when $n\to \infty$. 
\end{theorem}

\begin{proof}
Denote by $Z$ the subset of $\mathrm{Mat}_{n\times n}(2^{N_n})$ which consists of all matrices 
such that the corresponding multivalued binary operation $*$ for every $a,b,c\in N_n$ satisfies 
\begin{displaymath}
\bigcup_{s\in a*b}s*c=\bigcup_{t\in b*c}a*t=N_n.
\end{displaymath}
Then $Z\subset Y'\subset Y$ and hence it is enough to show that $\frac{|Z|}{2^{n^3}}\to 1$ when $n\to \infty$.

It is enough to show that for the subset $U$ of $\mathrm{Mat}_{n\times n}(2^{N_n})$ which 
consists of all matrices such that the corresponding multivalued binary operation $*$ for every $a,b,c\in N_n$ 
satisfies 
\begin{equation}\label{eq53}
\bigcup_{t\in b*c}a*t=N_n,
\end{equation}
we have $\frac{|U|}{2^{n^3}}\to 1$ when $n\to \infty$. Indeed, the latter, by symmetry, also implies
that $\frac{|U'|}{2^{n^3}}\to 1$ for $n\to \infty$, where $U'$ denotes the subset of $\mathrm{Mat}_{n\times n}(2^{N_n})$
corresponding to all operations $*$ satisfying
\begin{displaymath}
\bigcup_{s\in a*b}s*c=N_n 
\end{displaymath}
for all $a,b,c\in N_n$, so that we have $\frac{|U\cap U'|}{2^{n^3}}\to 1$ for $n\to \infty$.

Given a multivalued binary operation $*$ on $N_n$, with each $a\in N_n$ we can associate, similarly
to Subsection~\ref{s4.7}, a binary relation
$\tau_a$ on $N_n$ defined as follows: $i\,\tau_a\, j$ if and only if $i\in a*j$. Let $\mathbf{f}$
denote the full relation on $N_n$. Then \eqref{eq53} is equivalent to the fact that
$\tau_a\tau_b=\mathbf{f}$ for all $a,b\in N_n$.

To prove the latter claim we can adopt the classical argument from \cite[Theorem~4]{KR}. Our choice of an element
of $\mathrm{Mat}_{n\times n}(2^{N_n})$ corresponds to a random choice of $n$ binary relations 
$\tau_a$, $a\in N_n$, on $N_n$. We claim that the probability of the random event that the product of any two of these
$n$ relations equals the full relation tends to $1$ when $n$ tends to infinity (\cite[Theorem~4]{KR} claims this
just for two instead of $n$ random binary relations). Indeed, the probability that the product of two random
elements of the two-element Boolean algebra is zero equals $\frac{3}{4}$. This implies that the probability that the
element in a
fixed entry in a product of two random Boolean $n\times n$ matrices is zero is at most 
$\left(\frac{3}{4}\right)^{n}$. We have $n$ random elements $\tau_a$, $a\in N_n$. We can form
$n^2$ pairs $(\tau_a,\tau_b)$, $a,b\in N_n$, from these elements. The product 
$\tau_a\tau_b$  has $n^2$ entries, each of which 
is zero with probability at most $\left(\frac{3}{4}\right)^{n}$. Hence the probability that at least 
one of these entries in at least one of the products of the form $\tau_a\tau_b$ is zero is at most
$n^4\left(\frac{3}{4}\right)^{n}$. Since $n^4\left(\frac{3}{4}\right)^{n}\to 0$ for $n\to \infty$, the claim follows.
\end{proof}

\section{Nilpotent multisemigroups}\label{s7}

In this section we establish some basic fact about nilpotent multisemigroups (see e.g. \cite{GM,GM2} for more
advanced semigroup analogues).

\subsection{Nilpotent multisemigroups and their characterization}\label{s7.1}

Let $(S,*)$ be a multisemigroup. As usual, for an element $s\in S$ and $k\in\mathbb{N}$ we write
$s^k$ for the product $s*s*\dots *s$ of length $k$. Similarly, for $X\subset S$ we write $X^k$ for the union
of all $x_1*x_2*\dots *x_k$, where $x_i\in X$ for all $i$. An element $s\in S$ is said to be {\em nilpotent} 
provided that $s^k=\varnothing$ for some $k$. The multisemigroup $S$ is called {\em nilpotent} provided that 
$S^k=\varnothing$ for some $k$. The minimal such $k$ is called the {\em nilpotency degree} of $S$.
Similarly, a subset $X\subset S$ is called {\em nilpotent} provided that 
$X^k=\varnothing$ for some $k$ (in particular, the empty set is nilpotent). 
The notion of nilpotent multisemigroups generalizes that of nilpotent semigroups.
Nilpotent semigroups correspond, via the construction given at the end of Subsection~\ref{s4.2}, 
to nilpotent quasi-semigroups.

For a multisemigroup $(S,*)$ define the {\em action digraph} $\Gamma=\Gamma_S$ as follows: the set of vertices 
of $\Gamma$ is $S$ and for $s,t\in S$ (not necessarily different) we have an oriented edge $s\to t$ if and only 
if there exists $a\in S$ such that $t\in a*s$.

\begin{proposition}\label{prop71}
Let $(S,*)$ be a multisemigroup.
\begin{enumerate}[$($a$)$]
\item\label{prop71.1} The multisemigroup $(S,*)$ is nilpotent if and only if there is $m\in\mathbb{N}$ such 
that the length of any directed path in $\Gamma$ is smaller than $m$.
\item\label{prop71.2} If $(S,*)$ is nilpotent, then the nilpotency degree of $S$ is exactly the length of the
longest directed path in $\Gamma$ minus two.
\end{enumerate}
\end{proposition}

\begin{proof}
Let $(S,*)$ be nilpotent of nilpotency degree $k$ and $s_0\to s_1\to \dots \to s_m$ be an oriented path in
$\Gamma$ of length $m$. Then there exist $a_i$ such that $s_i\in a_i*s_{i-1}$, $i=1,2,\dots,m$. This means that
$a_m*a_{m-1}*\dots*a_{1}*s_0\neq\varnothing$ and hence $m+1<k$. 

On the other hand, if the length of any oriented path in $\Gamma$ is at most $m$, then for any
$m+2$ elements $s_1,\dots,s_{m+2}\in S$ we have $s_{m+2}*s_{m+1}*\dots *s_1=\varnothing$ and hence 
$S$ is nilpotent of nilpotency degree at most $m+2$. This implies both claims of the proposition.
\end{proof}

\begin{corollary}\label{cor72}
Let $(S,*)$ be a nilpotent multisemigroup. Then all Green's relations on $S$ coincide with the
equality relation.
\end{corollary}

\begin{proof}
This follows directly from the fact that $\Gamma$ does not have any oriented cycles. The latter is
a direct consequence of Proposition~\ref{prop71}.
\end{proof}

\subsection{Finite nilpotent multisemigroups}\label{s7.2}

Similarly to the case of finite nilpotent semigroups (see \cite[Chapter~7, Fact~2.30]{Ar}), for finite nilpotent
multisemigroups we have:

\begin{proposition}\label{prop73}
A finite multisemigroup  $(S,*)$ is nilpotent if and only if every element of $S$ is nilpotent.
\end{proposition}

\begin{proof}
The ``only if'' part is obvious, so we prove the ``if'' part. Assume that every element of $S$ is nilpotent.
We claim that the digraph $\Gamma$ does not have oriented cycles (in particular, loops). Indeed, assume that this 
is not the case and let $s_1\to s_2\to \dots \to s_k\to s_{k+1}=s_1$ be an oriented cycle in $\Gamma$. For 
$i=1,2,\dots, k$ let $a_i\in S$ be such that $s_{i+1}\in a_i*s_i$. Then there exists $t\in a_k*a_{k-1}*\dots *a_1$ such 
that $s_1\in t*s_1$. This implies that $s_1\in t^n*s_1$ for any $n\in \mathbb{N}$ and hence $t$ is not nilpotent,
a contradiction.

As $S$ is finite and $\Gamma$ does not have oriented cycles, then the length of any oriented path in $\Gamma$ is 
strictly smaller than $|S|$. This means that $S$ is nilpotent by Proposition~\ref{prop71}.
\end{proof}

Finite nilpotent semigroups can be characterized as finite semigroups with unique idempotent which, moreover, is
the zero element. Equivalently, a finite semigroup $(S,\cdot)$ is nilpotent if and only if for any non-singleton 
subsemigroup $T$  of $S$ we have $|T\cdot T|<|T|$.
The following gives an analogue of the latter characterization of nilpotency for multisemigroups.

\begin{theorem}\label{thm74}
A finite multisemigroup  $(S,*)$ is nilpotent if and only if for any (non-empty) submultisemigroup $T$ of $S$ we have
$|T*T|<|T|$ (or, equivalently, $T*T\neq T$).
\end{theorem}

\begin{proof}
If $S$ is nilpotent and $T$ is a submultisemigroup of $S$, then $T$ is nilpotent as well. However, $T*T=T$ implies
$T^k=T$ for all $k$ which contradicts nilpotency of $T$. Therefore $T*T\neq T$.

If $S$ is not nilpotent, we can use the fact that $S$ is finite and choose some non-nilpotent submultisemigroup $T$
of $S$ which is minimal with respect to inclusions. We claim that $T*T=T$. If this were not the case, then
$U:=T*T\neq\varnothing$ would be a submultisemigroup of $S$ properly contained in $T$. Hence $U$ must be
nilpotent by minimality of $T$. Therefore $U^k=\varnothing$ for some $k$, which implies that 
$(T*T)^k=T^{2k}=\varnothing$, a contradiction.
\end{proof}

\subsection{The radical}\label{s7.3}

Let $(S,*)$ be a multisemigroup. Following \cite{GM}, by the {\em radical} $R(S)$ of $S$ we will mean the set
\begin{displaymath}
R(S):=\{s\in S\,\vert \, S^1*s*S^1\text{ is nilpotent}\}.
\end{displaymath}
Then $R(S)$ is a union of two-sided ideals of $S$ and hence is a two-sided ideal of $S$.

\begin{lemma}\label{lem81}
If $S$ is finite, then  $R(S)$ is the maximal (with respect to inclusion) nilpotent two-sided ideal of $S$.
\end{lemma}

\begin{proof}
For $s\in R(S)$ we have that $S^1*s*S^1$ is nilpotent, in particular, $s$ is nilpotent. Hence every element of
$R(S)$ is nilpotent and thus the fact that $R(S)$ is nilpotent follows from Proposition~\ref{prop73} and the fact 
that $R(S)$ is finite (since $S$ is finite).

On the other hand, if $I$ is a nilpotent two-sided ideal of $S$ and $s\in I$, then $S^1*s*S^1\subset I$ is
nilpotent and hence $s\in R(S)$. The claim follows.
\end{proof}

Our next observation is the following:

\begin{proposition}\label{prop82}
Let $S$ be a finite multisemigroup and $T$ a nilpotent submultisemigroup of $S$.
\begin{enumerate}[$($a$)$]
\item\label{prop82.1} The set $T\cup R(S)$ is a nilpotent submultisemigroup of $S$.
\item\label{prop82.2} If $T$ is maximal with respect to inclusion, then $R(S)\subset T$. 
\end{enumerate}
\end{proposition}

\begin{proof}
That $T\cup R(S)$ is submultisemigroup follows directly from the facts that $T$ is a submultisemigroup and 
$R(S)$ is a two-sided ideal of $S$. That $T\cup R(S)$ is nilpotent follows from Proposition~\ref{prop73}.
This proves claim \eqref{prop82.1}. Claim \eqref{prop82.2} follows from claim \eqref{prop82.1}.
\end{proof}

\begin{corollary}\label{cor83} Let $S$ be a finite multisemigroup.
\begin{enumerate}[$($a$)$]
\item \label{cor83.1} If $S$ is nilpotent, then $S$ is the unique maximal 
(with respect to inclusion) nilpotent submultisemigroups of $S$.
\item \label{cor83.2} If $S$ is not nilpotent, then the map $X\mapsto X\hspace{-1mm}\setminus\hspace{-1mm} R(S)$
is a bijection from the set of all maximal nilpotent submultisemigroups of $S$
to the set of all maximal nilpotent submultisemigroups of $S\hspace{-0.7mm}\setminus\hspace{-1mm} R(S)$.
\end{enumerate}
\end{corollary}

\begin{proof}
Claim \eqref{cor83.1} is obvious. 

To prove claim \eqref{cor83.2}, let $X$ be a maximal nilpotent submultisemigroup of $S$.
Then $X\hspace{-1mm}\setminus\hspace{-1mm} R(S)$ is a nilpotent submultisemigroup of 
$S\hspace{-0.7mm}\setminus\hspace{-1mm} R(S)$. Assume $Y$ is a nilpotent submultisemigroup of 
$S\hspace{-0.7mm}\setminus\hspace{-1mm} R(S)$ containing $X\hspace{-1mm}\setminus\hspace{-1mm} R(S)$. 
Then from Proposition~\ref{prop82} it follows that $Y\cup R(S)$ is a nilpotent submultisemigroup of $S$. 
Therefore $Y=X\hspace{-1mm}\setminus\hspace{-1mm} R(S)$ by maximality of $X$, that is 
$Y$ is maximal. 

Conversely, let $Y$ be a maximal  nilpotent submultisemigroup of 
$S\hspace{-1mm}\setminus\hspace{-1mm} R(S)$. Then, similarly to the above, 
$Y\cup R(S)$ is a nilpotent submultisemigroup of $S$. If $X$ is a nilpotent submultisemigroup of $S$
containing $Y\cup R(S)$, then $X\setminus R(S)$ is a nilpotent submultisemigroup of 
$S\hspace{-0.7mm}\setminus\hspace{-1mm} R(S)$ containing $Y$. Hence $X=Y\cup R(S)$ by maximality of $Y$.
This proves claim \eqref{cor83.2}.
\end{proof}

\subsection{Maximal nilpotent submultisemigroups of strongly simple multisemigroups}\label{s7.4}

Let $(S,*)$ be a strongly simple multisemigroup which is not isomorphic to the singleton multisemigroup $\mathbf{0}$.
Recall that $\mathcal{H}$ is a congruence on $S$ and that $S/\mathcal{H}$ is a quasi-semigroup.

\begin{proposition}\label{prop85}
The canonical strong surjective homomorphism $\varphi:S\tto S/\mathcal{H}$ induces a bijection between the set of maximal (with respect to 
inclusion) nilpotent submultisemigroups of $S$ and maximal nilpotent sub-quasi-semigroups of $S/\mathcal{H}$. This
bijection preserves nilpotency degree.
\end{proposition}

\begin{proof}
From Lemma~\ref{lem23} we have that for $a,b\in S$ the product $a*b$ is non-empty if and
only if the product $c*d$ is non-empty for all $c\in\mathcal{H}_a$ and $d\in\mathcal{H}_b$. Furthermore, 
from Proposition~\ref{prop25} it also follows that $s\,\mathcal{H}\,t$ for any $s\in a*b$
and $t\in c*d$. Therefore, if $T$ is a nilpotent submultisemigroup of $S$, then 
$\displaystyle T':=\bigcup_{t\in T}\mathcal{H}_t$ is a nilpotent submultisemigroup of $S$ of the
same nilpotency degree, moreover, $T\subset T'$. Hence any maximal nilpotent submultisemigroup of $S$
is a union of $\mathcal{H}$-classes. If $H$ is an $\mathcal{H}$-class of a maximal nilpotent submultisemigroup of $S$,
then the above arguments imply $H*H=\varnothing$. Both claims of the proposition follow readily from these observations.
\end{proof}

Let $(S,*)$ be a finite strongly simple multisemigroup, $I_1,\dots, I_m$ --- the list of all  ${\mathcal L}$-classes, and $J_1,\dots, J_n$ --- the list of all  ${\mathcal R}$-classes in $S$. Let $\mathbf{I}(S)$ denote the Boolean matrix $(a_{ij})$, $i=1,\dots, m$, $j=1,\dots, n$, defined as follows:
\begin{displaymath}
a_{ij}=\left\lbrace
\begin{array}{ll} 1, & I_i\cap J_j \text{ is a hypergroup; }\\
0,& \text{ otherwise. }\end{array}
\right.
\end{displaymath}
The matrix $\mathbf{I}(S)$ is called the {\em incidence matrix} of $S$. Recall that $S\not\cong\mathbf{0}$. Then,
by Theorem \ref{thm31}, all rows and columns of $S$ are non-zero.
Assume that $\mathbf{I}(S)$ has the form
\begin{displaymath}
\left(\begin{array}{c|c}E_k&B\\\hline  A&A\cdot B\end{array}\right) 
\end{displaymath}
for some positive integer $k$, where $E_k$ is the identity matrix and
$A$ and $B$ are Boolean matrices such that each row of $A$ is not zero and each column of $B$ is not zero,
and $A\cdot B$ is the Boolean matrix which is the product of $A$ and $B$. In this case a classification of maximal nilpotent sub-quasi-semigroups of $S/\mathcal{H}$ can be found in \cite[Subsection~6.10]{GM}. In particular, there
are exactly $k!$ such sub-quasi-semigroups and each of them has nilpotency degree $k$.

\vspace{1cm}

\noindent
G.K.: Faculty of Computer and Information Science, University of Ljubljana,  
Tr{\v z}a{\v s}ka cesta 25, SI-1001, Ljubljana, SLOVENIA, 
e-mail: {\tt ganna.kudryavtseva\symbol{64}fri.uni-lj.si}
\vspace{0.5cm}

\noindent
V.M.: Department of Mathematics, Uppsala University, SE 471 06,
Uppsala, SWEDEN, e-mail: {\tt mazor\symbol{64}math.uu.se}

\end{document}